\documentclass{article}
\usepackage[T1]{fontenc}
\usepackage{fancyhdr}
\usepackage{amsmath}
\usepackage{amssymb}
\usepackage{amsthm}
\usepackage{enumerate}
\usepackage{epsfig}
\usepackage{graphicx}
\usepackage{xcolor}
\usepackage{url}
\usepackage{exscale}
\usepackage{geometry}
\usepackage[textwidth=4cm,textsize=footnotesize]{todonotes}
\usepackage{fancyhdr}
\usepackage{booktabs}

\theoremstyle{plain} \newtheorem{theorem}{Theorem}[section]
\newtheorem{proposition}[theorem]{Proposition}

\newtheorem{lemma}[theorem]{Lemma}

\theoremstyle{remark}

\newcommand{\mean}{{\mathbf E}}
\newcommand{\var}{{\rm var}}

\newcounter{hypH}

\newcounter{hypCov}

\newcommand{\rmd}{\mathrm{d}}

\newcommand{\eqdef}{\ensuremath{:=}}

\newcommand{\idnonzero}{\kappa}

\newcommand{\dconv}{\stackrel{\mathcal{D}}{\rightarrow}}

 \def\doublespace{\baselineskip=\normalbaselineskip \multiply\baselineskip by
 3 \divide\baselineskip by 2}

\begin{document}
%%%%%%%%%%%%%%%%%%%%%%%%%%%%%%%%%%%%%%%%%%%%%%%%%%%%%%%%%%%%%%%%%%%%%%5

\doublespace

\title{Statistical Inference for Oscillation Processes}
\date{}
\author{Rainer Dahlhaus\footnote{Institute of Applied Mathematics, Heidelberg University, Germany.}, Thierry Dumont\footnote{MODAL'X, Universit\'e Paris-Ouest, Nanterre, France.}, Sylvain Le Corff\footnote{Laboratoire de Math\'ematiques d'Orsay, Univ. Paris-Sud, CNRS, Universit\'e Paris-Saclay, 91405 Orsay, France.}, Jan C. Neddermeyer\footnote{DZ BANK AG, Frankfurt, Germany.}}

\lhead{Dahlhaus et al.}
\rhead{Statistical Inference for Oscillation Processes}

\maketitle

\begin{abstract}
A new model for time series with a specific oscillation pattern is proposed. The model consists of a hidden phase process controlling the speed of polling and a nonparametric curve characterizing the pattern, leading together to a generalized state space model. Identifiability of the model is proved and a method for statistical inference based on a particle smoother and a nonparametric EM algorithm is developed. In particular, the oscillation pattern and the unobserved phase process are estimated. The proposed algorithms are computationally efficient and their performance is assessed through simulations and an application to human electrocardiogram recordings.
\end{abstract}

\medskip


\section{Introduction} \label{sec:Introduction}
In this paper, we propose a model for the statistical analysis of oscillating time series. In its simplest form, the model is a generalized state space model (GSSM) with nonlinear observation equation
\begin{equation} \label{BasicModel}
Y_t = f(\phi_t) + \varepsilon_t, \quad t \in \mathbb{N}_+ \quad \mbox{with} \quad \varepsilon_t \stackrel{\text{iid}}{\sim} \mathcal{N}(0,\sigma_{\varepsilon}^{2})\,,
\end{equation}
where $f$ is an unknown $2\pi$-periodic function (the oscillation pattern), $\{\phi_t\}_{ t \in \mathbb{N}_+}$ is an unobserved stochastic phase process (the internal clock of the oscillator), and the $\{\varepsilon_t\}_{t \in \mathbb{N}_+}$ are independent of the process $\{\phi_t\}_{ t \in \mathbb{N}_+}$. The phase process $\{\phi_t\}_{t}$ will be an integrated process which may either slow down or speed up the cycle, for example $\phi_t$ may be modeled by the state equation
\begin{equation} \label{StateEquation1}
\Delta \phi_t = \omega +  \eta_t, \quad t \in \mathbb{N}_+  \quad \mbox{with} \quad \eta_t \stackrel{\text{iid}}{\sim} \mathcal{N}(0,\sigma_{\eta}^{2}), \quad \omega \in (0,\pi)\,,
\end{equation}
where $\Delta \phi_t := \phi_{t}-\phi_{t-1}$, and $\phi_0$ is assumed to be uniformly distributed on $[0,2\pi]$ independent of the other variables. In this case the $\Delta \phi_t$ are iid and $\phi_t = \phi_0 + \omega \,t +  S_t$ with $S_t := \sum_{s=1}^{t} \eta_s$ being a random walk with variance $\var(S_t) = \sigma_{\eta}^{2} \,t$.

It is informative to compare the above model with the model $\phi_t = \phi_0 + \omega \,t + \mathsf{R}_t$ where $\mathsf{R}_t$ is not integrated but stationary. These two models reflect two completely different situations with different type of oscillators: The latter can be used as a model where the oscillator sticks except from small deviations to an external ``pacemaker'' (say where a hormone is driven by the circadian rhythm), while the integrated phase model of this paper is e.g. a suitable model for ECG data (cf. Figure~\ref{fig:ObservationsECG} - to be analyzed in Section~\ref{sec:HumanECG}) and for the data of a R\"{o}ssler attractor (cf. Figure~\ref{fig:DataRoesslerAttractor} - to be analyzed in Section~\ref{sec:NoisyRoessler}) where one does not know in advance at which point of the cycle we  will be at a specific time point. This is reflected by a constant variance of $\phi_t$ in the stationary model and an increasing variance of $\phi_t$ in the present model (\ref{StateEquation1}).

We also mention the classical ``hidden frequency'' model of Hannan (1973) and others (for an overview see Quinn and Hannan (2001) and Li (2013)), which in its simplest form is $Y_t = A \cos (\omega t) + \varepsilon_t$ with a stationary process $\varepsilon_t$. In the notation of this paper this work dealt with the case $f(\cdot) = \cos (\cdot)$ and $\phi_t = \omega t$. There exist many papers with rather deep mathematical results on the estimation of $\omega$ in this framework - in particular via maximization of the periodogram which originally dates back to Schuster (1897). The case of a time varying hidden frequency $\omega_t$ is more in the spirit of this paper. In this case one may use the maximization of the periodogram on rolling data windows (c.f. Paraschakis and Dahlhaus, 2011).

\begin{figure}
\centering
\includegraphics[width=400pt, height=130pt]{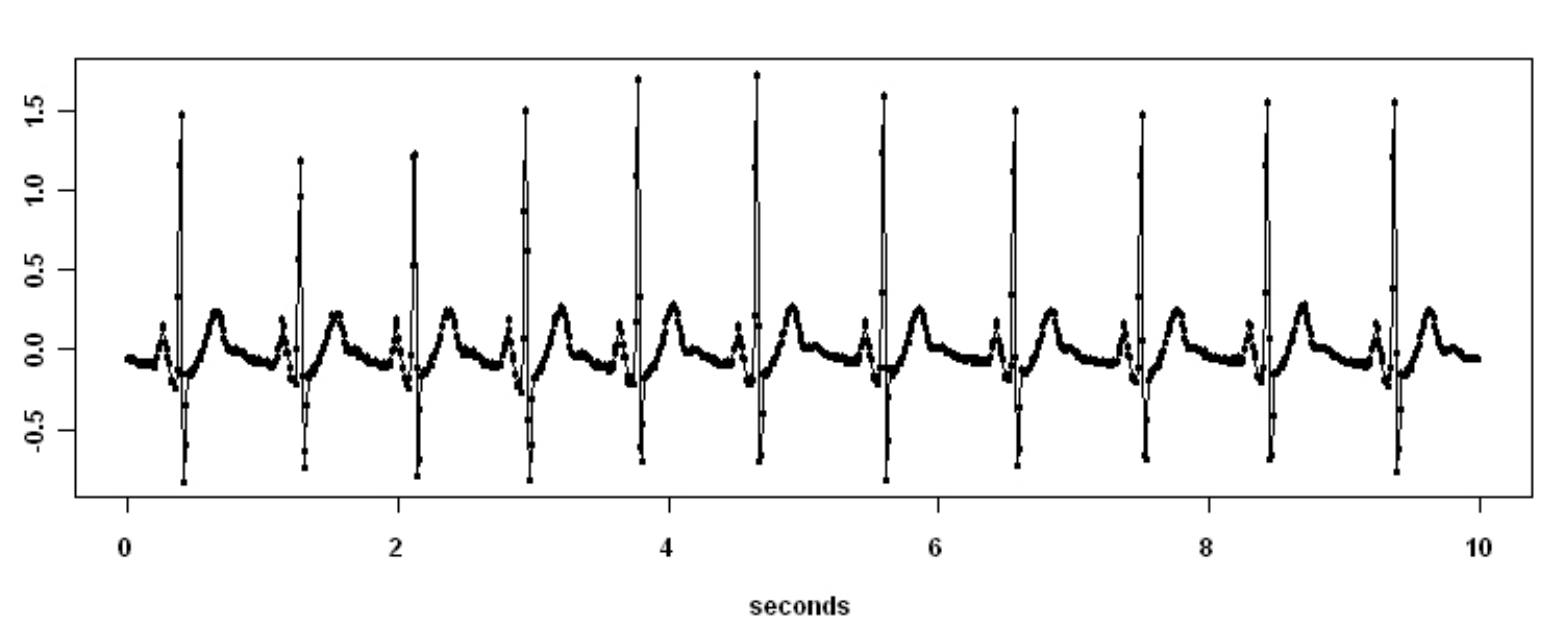}
\caption{\footnotesize 1000 Observations from an ECG}
\label{fig:ObservationsECG}
\end{figure}

\begin{figure}
\centering
\includegraphics[width=400pt, keepaspectratio]{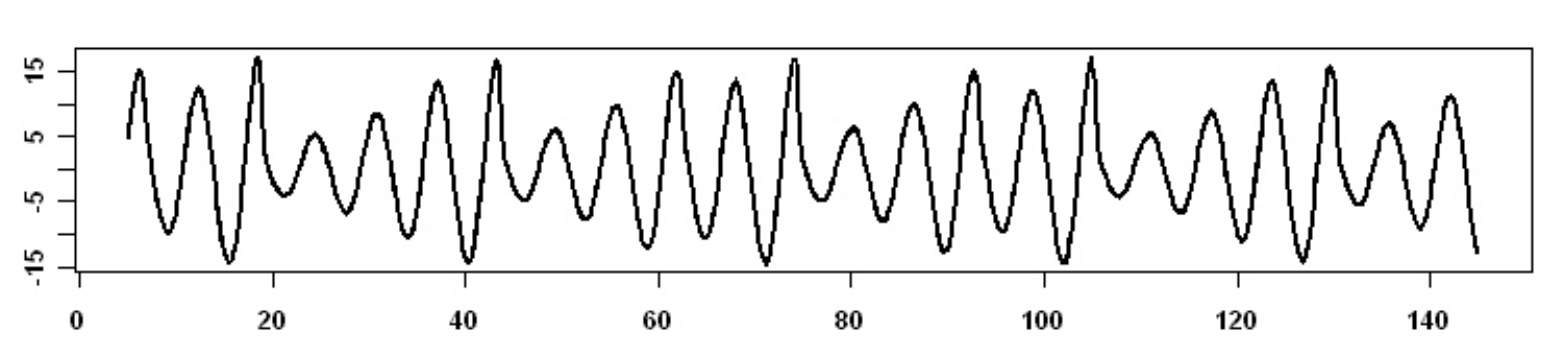}
\caption{\footnotesize 1415 Observations from a R\"{o}ssler Attractor}
\label{fig:DataRoesslerAttractor}
\end{figure}

Statistical inference for the above model is challenging since (\ref{BasicModel}) is a nonparametric regression model with unknown regressors (but, due to periodicity, with asymptotically infinite many replications). In Section~\ref{sec:Identifiability} and Section~\ref{Appendix:Identifiability}, we prove identifiability of model (\ref{BasicModel}), (\ref{StateEquation1}). Dumont and Le Corff (2014) have proved identifiability for a different nonparametric GSSM.

In Section~\ref{sec:SSM} we then show how the unknown oscillation pattern $f$ and the unobserved phase process $(\phi_t)_{\! t}$ can be estimated / predicted. In that section we consider the more general model
\begin{equation}\label{intro:model}
Y_t = a_t f(\phi_t) + b_t + \varepsilon_t\,,
\end{equation}
which includes a time varying amplitude $a_t$ and a time varying baseline $b_t$. We also allow for the more general state equation (\ref{intro:acdstateequation}) which guarantees positivity of the phase increments and includes some dependence of the phase increments. Maximum likelihood based estimation of $\phi_t$, $a_t$ and $b_t$ requires computation of the smoothing distributions, i.e. the posterior distributions of the sequence of hidden states given all observations $Y_0,\ldots, Y_T$. These distributions cannot be computed explicitly using the GSSM model of this paper, and we propose in Section~\ref{sec:SSM} an efficient (fixed-lag) Rao-Blackwellized particle smoother that combines the Kalman smoother and the fixed-lag particle smoother introduced and analyzed in Olsson et al. (2008). Estimates of $\phi_t$, $a_t$ and $b_t$ can then be obtained as means of the smoothing distributions. In Section~\ref{sec:EM_Algorithm} and \ref{sec:NonparametricEstimate}, a nonparametric EM algorithm is  developed for the estimation of the function $f$ and other parameters. For a recent overview of sequential Monte Carlo methods see Douc et al. (2014).

As a by-product we also obtain a method for nonlinear phase estimation in the case where $f(x) = \cos(x)$ is known. Such an estimator is needed in several applications - for example for the detection of phase synchronization of chaotic oscillators (cf. Pikovsky et.al., 2001); in neuroscience for the investigation of functional coupling of different brain regions (cf. Fell and Axmacher, 2011), and in engineering for channel decoding (Chen et.al., 2003). Furthermore, the estimation of instantaneous frequencies is a key step in the widely used empirical mode decomposition (cf. Huang et.al., 1998).

Traditional approaches for phase estimation are based on the Hilbert transform (cf. Pikovsky et.al., 2001) or the Wavelet transform (c.f. Grossmann et.al., 1989). In practice, these methods often fail not only in situations of fast varying frequencies but also in cases where the
signal is corrupted by noise. In such a situation statistical methods such as variants of Hannan's hidden frequency approach or the method of this paper are beneficial. Phase estimation with a GSSM has been used in Tsakonas et.al. (2008), in the context of communication systems in Amblard (2003), and in the context of audio processing in Dubois and Davy (2007) - in all cases in combination with a different  model for the state and only for $f(x) = \cos(x)$ (i.e. in particular under the assumption that $f$ is known).

As mentioned above, Section~\ref{sec:Identifiability} contains identifiability and Section~\ref{sec:SSM} statistical inference in the framework of a GSSM. Nonlinear phase estimation in the special case $f(\cdot) = \cos(\cdot)$ is explored in a simulation study in Section~\ref{sec:SimulatedData} and for the nonlinear R\"{o}ssler-attractor in Section~\ref{sec:NoisyRoessler} followed by an application to ECG data in Section~\ref{sec:HumanECG}. Section~\ref{sec:Conclusion} contains concluding remarks and the appendices in Section~\ref{Appendix:Identifiability} and \ref{sec:Appendix2} the proofs.

\section{Identifiability of Oscillation Processes} \label{sec:Identifiability}

Since both the function $f$ and the phases $\phi_t$ are unobserved, identifiability is a nontrivial  issue discussed in this section. One comment beforehand: as usual identifiability means identifiability of $f$ and the parameters of the process $\phi_t$ and not ``identifiability'' of $\phi_t$. The latter is not possible - the best we can achieve is e.g. the determination of the best predictor of $\phi_t$ given the observations. We start with a heuristic discussion of the identifiability problems:

\begin{enumerate}[1)]
\itemsep-0.05cm
\item \label{enumerate1:shifting} \underline{Shifting the oscillation pattern}:\; it is obvious that the starting point of the oscillation pattern cannot be identified. Formally,
\begin{equation*} \label{}
f(\phi_t) = \tilde{f} (\tilde{\phi}_t) \quad \mbox{with} \quad \tilde{f} (x) = f(x - \theta) \quad \mbox{and} \quad \tilde{\phi}_t = \phi_t + \theta\,.
\end{equation*}
Sometimes there exists a natural starting point known from the applied problem at hand - in other cases one can just define the starting point arbitrarily.

\item \label{enumerate1:omega} \underline{Non-identifiability of the mean phase increment}:\;
the mean phase increment $\omega$ can only be identified from the data under additional assumptions. For example, for the model (\ref{StateEquation1}) where $\phi_t = \phi_0 + \omega \,t + \sigma \, S_t$, we have
\begin{equation*} \label{}
f(\phi_t) = f(\phi_0 + \omega \,t + \sigma \, S_t) = f(\phi_0 + (\omega + 2 \pi \ell) \,t + \sigma \, S_t)\,,
\end{equation*}
i.e. $\omega$ is only identifiable $\mbox{mod}\;2\pi$. Furthermore, if $f$ is symmetric then
\begin{equation*} \label{}
f(\phi_t) = f(-\phi_t) = f(-\phi_0 - \omega \,t - \sigma \, S_t) = f(-\phi_0 + (2\pi -\omega) \,t - \sigma \, S_t)\,,
\end{equation*}
i.e. $\omega$ and $2\pi - \omega$ can only be identified under the additional assumption that $f$ is not symmetric or the distribution of  $\eta_t$ is not symmetric. Since we are mainly interested in the case where each cycle contains several data points (corresponding to a small $\omega$), we assume throughout this paper that  $\omega \in (0,\pi)$ guaranteeing identifiability.

Note that the above discussion is related to the Nyquist frequency and aliasing: If we regard the original signal as continuous in time with phase process $\phi_t = \phi_0 + \omega \,t + \sigma_{\eta} \, B(t)$ where $B(t)$ denotes a Brownian motion, and we sample the process at times $\Delta t$, then the same arguments as above show that we can only identify $\omega \in \big(0, \pi/\Delta\big)$ with aliases at $2\pi /\Delta - \omega$ and all frequencies shifted by $2 \pi \ell / \Delta, \ell \in \mathbb{Z}$.

\item \label{enumerate1:basiccycle} \underline{Minimal period length / the basic cycle}:\; let
\begin{equation} \label{DefinitionReplication}
\textnormal{\small{repl}}(f) \eqdef \sup \Big\{ r \in \mathbb{R} \, \Big| \,f\Big(\frac {\cdot} {r}\Big) \mbox{ is $2\pi$ periodic} \Big\}
\end{equation}
be the number of periodic replications in $f$. Then $\tilde{f} (\cdot) := f\big(\frac {\cdot} {\textnormal{\scriptsize{repl}}(f)}\big)$ is called the \underline{basic cycle} of the oscillation. In Theorem~\ref{TheoremIdentifiability2}, we prove that $\mbox{\small{repl}}(f) \!\in\! \mathbb{N}_+$ and that the basic cycle is unique up to the shift of the starting point. Thus, if we fix the starting point $\theta_{\! f}$ in the basic cycle we have
\begin{equation} \label{DefinitionBasicCycle}
f_{\mbox{\scriptsize{basic}}} (x) : = f\Big(\frac {x- \theta_{\!f}} {\textnormal{\scriptsize{repl}}(f)}\Big)
\end{equation}
with a unique $f_{\mbox{\scriptsize{basic}}}$. If we have two representations with different $f_1$ and $f_2$ then
\begin{equation*} \label{}
f_1\Big(\frac {x - \theta_{1}} {\textnormal{\small{repl}}(f_1)}\Big) = f_{\mbox{\scriptsize{basic}}} (x) = f_2\Big(\frac {x - \theta_{2}} {\textnormal{\small{repl}}(f_2)}\Big)\quad \mbox{i.e.}\quad   f_1(x) = f_2(\frac {x-\theta} {\gamma})\,,
\end{equation*}
with $\gamma \!=\! \textnormal{\small{repl}}(f_2) / \textnormal{\small{repl}}(f_1)$ and $\theta \!=\! (\theta_2 \!-\! \theta_1) / \textnormal{\small{repl}}(f_1)\;$ (see Theorem~\ref{TheoremIdentifiability} and \ref{TheoremIdentifiability2} below).

It is important to note how the phases transform when moving from the oscillation pattern $f_1$ to $f_2$. If $\{\phi_t^{(1)}\}$ fulfills model (\ref{StateEquation1}) we have
\begin{equation*} \label{}
f_1(\phi_t^{(1)}) = f_2 (\phi_t^{(2)}) \quad \mbox{with}\quad   \phi_t^{(2)} = \frac {\phi_t^{(1)}-\theta} {\gamma} \quad \mbox{i.e.} \quad  \Delta \phi_t^{(2)} = \frac {\Delta \phi_t^{(1)}} {\gamma} = \frac {\omega_1} {\gamma} + \frac {\sigma_1} {\gamma} \,  \eta_t\,.
\end{equation*}
Thus  $\{\phi_t^{(2)}\}$ also fulfills model (\ref{StateEquation1}) with $\gamma \omega_2 \!=\! \omega_1$ and $\gamma \sigma_2 \!=\! \sigma_1$ (see also Theorem~\ref{TheoremIdentifiability}). For identifiability reasons, we usually assume that the period length is the minimal one, i.e. we use the basic cycle as our oscillation pattern.

In practice, the discrimination between the basic cycle and multiple replications is often clear from eye-inspection (as in Figure~\ref{fig:ObservationsECG}). One can incorporate this external information into the EM algorithm from Section~\ref{sec:NonparametricEstimate} in a quite elegant way: as demonstrated above the information on the multiplicity of the cycle (say $f(x) \!=\! f_{\mbox{\scriptsize{basic}}}(rx)$ with $r \! \in \! \mathbb{N}_+$) is also contained in the drift-parameter which becomes $\omega \!=\! \omega_{\mbox{\scriptsize{basic}}} / r$. Incorporating external information can then be achieved by choosing an appropriate starting value for $\omega$ in the algorithm (heuristically the EM algorithm then finds that local maximum which corresponds to the basic cycle). In practice, we may count the number of basic cycles in the data, multiply it by $2\pi$ and divide it by the number of time points leading to a rough estimate of $\omega$ which we use as the initial value. The ECG-example in Section~\ref{sec:HumanECG} shows that this works remarkably well - even with the uninformative starting value $f^{(0)}\equiv 0$. The starting value for $\omega$ then allows for a reasonable simulation of $\phi_t$ (the particles) leading to improved estimates of $f$ in the next iteration steps.

\item \label{enumerate1:timewarping} \underline{Time-warping}:\; non-identifiability due to time-warping seems to be a serious problem. Time warping means a transformation of the observation model of the form
\[
y_t =  (f \circ g^{-1})\big( g(\phi_t \; \mbox{mod}\;2\pi)\big) + \varepsilon_t\,,
\]
with an increasing function $g\!:\! [0,2\pi] \!\rightarrow \![0,2\pi]$ leading to the new oscillation pattern $\tilde{f} = f \circ g^{-1}$ and the new phases $\tilde{\phi}_t = g(\phi_t\; \mbox{mod}\;2\pi)$. The problem of time-warping has been discussed in nonparametric regression in a large number of papers (cf. Kneip and Gasser, 1992; Wang and Gasser, 1997).

Luckily, the present model rules out time warping in a very natural way: if $\phi_0 \sim U[0,2 \pi]$ and the increments $\Delta \phi_t$ are independent of $\phi_0$ then $\phi_t \; \mbox{mod}\;2\pi \;  \sim U[0,2 \pi]$. However, this is only true for $\tilde{\phi}_t$ if $g(x)=x$, i.e. time warping is automatically ruled out. If we abstain from the assumption $\phi_0 \sim U[0,2 \pi]$, then the assumption of stationarity of the increments implies that $\phi_t \; \mbox{mod}\;2\pi \;  \dconv U[0,2 \pi]$, meaning that the assumption of stationarity of the increments prevents time-warping.

    \newcounter{enumTemp}
    \setcounter{enumTemp}{\theenumi}
\end{enumerate}

\noindent In the more complicated model $y_t = a_t f(\phi_t) + b_t + \varepsilon_t$ with  amplitude $a_t$ and baseline $b_t$ (modeled by stochastic processes) there arise two additional identifiability issues.

\begin{enumerate} [1)]
\itemsep-0.05cm
\setcounter{enumi}{\theenumTemp}
\item \label{enumerate1:amplitude} \underline{Amplitude of $f$}:\; in case of a time varying amplitude $a_t$, we remove non-identifiability by assuming $\mean a_t = 1$, i.e. $a_t$ measures the relative deviation of the amplitude over time; in case of a known oscillation pattern (e.g. $f(x) = \cos (x)$), we make no assumption on $\mean a_t$.

\item \label{enumerate1:baseline} \underline{Level of $f$}:\; in case of a time varying baseline $b_t$, we remove non-identifiability by assuming $\mean b_t = 0$ i.e. $b_t$ measures the deviation of the baseline over time; in case of a known oscillation pattern, we make no assumption on $\mean b_t$.
\end{enumerate}

\setlength{\textheight}{220mm}%{240mm}
\setlength{\textwidth}{6.1in}%{5.9in}
\setlength{\topmargin}{-10mm}% for pdf

\noindent We now prove identifiability in a strict sense of model (\ref{BasicModel}), (\ref{StateEquation1}). In addition to the assumptions stated in Section~\ref{sec:Introduction} we assume that $f$ is a $2\pi$ periodic function with $f \in \mathcal{F}$, where
\begin{equation*} \label{}
\mathcal{F} \eqdef \bigg\{f:\mathbb{R}\to\mathbb{R}\; \Big|\; f(x) = \sum_{k\in\mathbb{Z}}c_k \mathrm{e}^{ikx} \quad  \mbox{with}  \quad \{c_k\} \in\ell_2(\mathbb{Z})\; \bigg\}\,.
\end{equation*}

\begin{theorem} \label{TheoremIdentifiability}
%Suppose there exists a non-constant oscillation pattern $f_{\star}$ and $\omega \!\in\! (0,\pi)$, $\sigma_{\varepsilon}^{2}$, $\sigma_{\eta}^{2}, f \!\in \!\mathcal{F}$ and $\omega_\star \!\in\! (0,\pi),\sigma_{\varepsilon \star}^{2},\sigma_{\eta \star}^{2}, f_\star \!\in \!\mathcal{F}$ are two parameter sets of the state space system (\ref{GSSM-Identifiability}) such that the finite dimensional distributions of $\{Y_k\}_{k \in \mathbb{N}}$ are the same. Under the additional assumption that $\omega_{\star}\notin \pi\mathbb{Q}$ there exists a
%$\gamma\in\mathbb{Q}$ and a $\theta\in[0,2\pi)$ such that $\sigma_{\varepsilon} = \sigma_{\varepsilon \star}$, $\gamma \sigma_{\eta} = \sigma_{\eta \star}$, $\gamma \omega = \omega_{\star}$, and $f(\frac {x-\theta} {\gamma}) = f_{\star}(x)$ for all $x\in\mathbb{R}$.
Let $f_{\star}\in\mathcal{F}$ be a non-constant oscillation pattern and $\omega_\star \!\in\! (0,\pi),\sigma_{\varepsilon \star}^{2},\sigma_{\eta \star}^{2}$ be the parameter set of the state space model (\ref{BasicModel}), (\ref{StateEquation1}).
Assume that there exist $f \!\in \!\mathcal{F}$ and $\omega \!\in\! (0,\pi)$, $\sigma_{\varepsilon}^{2}$, $\sigma_{\eta}^{2}$ such that the finite dimensional distributions of $\{Y_k\}_{k \in \mathbb{N}_+}$ given by the two parameter sets are the same. Under the additional assumption that $\omega_{\star}\notin \pi\mathbb{Q}$, there exists a
$\,\gamma\in\mathbb{Q}$ and a $\,\theta\in[0,2\pi)$ such that $\sigma_{\varepsilon} = \sigma_{\varepsilon \star}$, $\gamma \sigma_{\eta} = \sigma_{\eta \star}$, $\gamma \omega = \omega_{\star}$, and $f(\frac {x-\theta} {\gamma}) = f_{\star}(x)$ for all $x\in\mathbb{R}$.
\end{theorem}

The proof is given in Section~\ref{Appendix:Identifiability}. We just mention here that we can identify $\sigma_{\varepsilon}$, $\sigma_{\eta}$, $\omega$, and the squared Fourier-coefficients of $f$ from the second order structure of $\{Y_k\}$ while for the identification of $f$ higher order moments are needed. The assumption $\omega_{\star}\notin \pi\mathbb{Q}$ is dispensable in our opinion but we were unable to prove the result without it. In order to gain a deeper insight into the situation and to explain the constant $\gamma$ we define
\begin{equation} \label{DefFourierCoefficient}
c_k(f) \eqdef \frac{1}{2\pi}\int_{-\pi}^{\pi}f(x)\, \mathrm{e}^{-ikx}\,\rmd x\;
\end{equation}
and the sequence $\{\idnonzero_i (f)\}_{i\ge 1}$ by
\begin{equation} \label{DefKappas}
\idnonzero_1(f) = \inf \left\{k\ge 1; c_k(f)\neq 0\right\} \quad\mbox{and}\quad\idnonzero_{i+1}(f) = \inf \left\{k\ge \idnonzero_i(f)+1; c_k(f)\neq 0\right\}\;.
\end{equation}

\begin{theorem} \label{TheoremIdentifiability2}
Assume that there exist an oscillation pattern $f$ and $\omega \!\in\! (0,\pi)$, $\sigma_{\varepsilon}^{2}$, $\sigma_{\eta}^{2}$ such that the conditions of Theorem~\ref{TheoremIdentifiability} hold. Then, the basic cycle $f_{\mbox{\scriptsize{basic}}} (\cdot)$ defined in (\ref{DefinitionBasicCycle}) is unique and every oscillation pattern is an $\ell$-times replication of the basic cycle with $\ell \in \mathbb{N}_+$. Thus also $\mbox{\small{repl}}(f) \in \mathbb{N}_0$. Furthermore,
\begin{equation} \label{DefinitionReplication2}
\textnormal{\small{repl}}(f) =  \max \big\{ \ell \in \mathbb{N}_+ \,\big| \,c_k(f) = 0 \;\; \forall \; k \neq \ell\mathbb{N}_+ \big\}\,,
\end{equation}
and we have for $\gamma$ from Theorem~\ref{TheoremIdentifiability}
\begin{equation} \label{RelationGamma2}
\gamma = \frac {\textnormal{\small{repl}}(f)} {\textnormal{\small{repl}}(f_\star)} \qquad \mbox{and}\qquad \gamma = \frac {\idnonzero_i(f)} {\idnonzero_i(f_\star)} \;\; \mbox{for all}  \;i \in \mathbb{N}_+\,.
\end{equation}
In addition, the $\{\idnonzero_i(f_{\mbox{\scriptsize{basic}}})\}_{i \in \mathbb{N}_+}$ are setwise coprime.
\end{theorem}
Note that $\textnormal{\small{repl}}(f)$ is not necessarily equal to $\idnonzero_1(f)$. An example is $f(x) \!:=\! \cos (2x) \!+\! \cos (3x)$.

\section{Statistical Inference for Oscillation Processes} \label{sec:SSM}
As our model  is a GSSM statistical inference may be performed using Sequential Monte Carlo techniques. In the following,  the estimation of the oscillation pattern $f$ and of the other parameters uses a fixed-lag particle smoother in combination with a nonparametric MCEM algorithm.

\subsection{The Model} \label{sec:PracticalModel}

For practical purposes we modify/generalize the model (\ref{BasicModel}), (\ref{StateEquation1}) in two aspects.

1) In order to guarantee positivity of the increments, we use for the phase differences $\Delta{\phi_{t}}$ the ACD(1,0) (autoregressive conditional duration) model:
\begin{equation} \label{acd:model}
\Delta{\phi_{t}} = \big( \alpha + \beta \Delta{\phi_{t-1}} \big) \,\eta_t\,,
\end{equation}
where the $\eta_t$ are {e.g.} Beta or Gamma distributed with $\mathbf{E}\eta_t = 1$ (if $\mathbf{E}\eta_t \neq 1$, then $\eta_t$ can be replaced by $\tilde{\eta}_t = \eta_t/\mathbf{E}\eta_t$). We assume that $\alpha,\beta > 0$, $\beta < 1$, and $\alpha < \pi (1-\beta)$. It then can be shown that the (unconditional) mean of the phase increments is
\begin{equation}\label{acd:mean}
\omega = \mathbf{E}[\Delta{\phi_{t}}] = \frac{\alpha}{1-\beta} < \pi\,.
\end{equation}
The ACD model was originally introduced by Engle and Russell (1998) as a model for the dependency of the durations between consecutive transactions in financial markets.

2) For data such as the ECG data we add a time varying amplitude $a_t$ and a time varying baseline $b_t$ as hidden states. In order to keep the computations simple we use a Gaussian random walk model. This has in particular the advantage that we can use a
(fixed-lag) Rao-Blackwellized particle smoother that combines a sequential particle smoother for   $\phi_t$ with a Kalman smoother for $a_t$ and $b_t$ (Doucet et.al., 2000b). Estimates of $\phi_t$, $a_t$ and $b_t$ can then be obtained as the means of the smoothing distributions. For a recent overview of sequential Monte Carlo methods see Douc et al. (2014). Therefore, we use the observation model:%Summarizing we use the GSSM with state $\mathbf{x}_t = (a_t, b_t, \phi_t, \psi_t)^T$:
\begin{equation}
\label{intro:observationequation}
y_t \; = \;a_t \, f(\phi_t) + b_t + \varepsilon_t\,,
\end{equation}
with
\begin{equation}
\label{intro:acdstateequation}
\begin{pmatrix}
\phi_t \\
\psi_t
\end{pmatrix}  = \begin{pmatrix}
\phi_{t-1} + (\alpha + \beta \psi_{t-1}) \eta_t \\
(\alpha + \beta \psi_{t-1}) \eta_t\,.
\end{pmatrix}
\quad\mbox{and}\quad
\bigg[\!
\begin{pmatrix}
a_t\\
b_t
\end{pmatrix} \!-\!
\begin{pmatrix}
\mu_a\\
\mu_b
\end{pmatrix}
\!\bigg] =
A \,
\bigg[\!
\begin{pmatrix}
a_{t-1}\\
b_{t-1}
\end{pmatrix} \!-\!
\begin{pmatrix}
\mu_a\\
\mu_b
\end{pmatrix}
\!\bigg]
+
\begin{pmatrix}
\xi_t\\
\zeta_t
\end{pmatrix}\,,
\end{equation}

\noindent where $(\xi_t, \zeta_t)^T \sim \mathcal{N}(\mathbf{0}, Q)$ and $\varepsilon_t \sim
\mathcal{N}(0, \sigma_{\varepsilon }^2)$. It is assumed that $\varepsilon_t$, $\eta_t$ and
$(\xi_t, \zeta_t)^T$ are mutually and serially independent. For simplicity we assume
that $A=\mbox{diag}(1,1)$ (this is a typical trend model) and  $Q$ is diagonal. We assume that $\mu_a=1$ and $\mu_b=0$ in case where the oscillation pattern $f$ is a nonparametric function. If (say) $f(\cdot) = \cos (\cdot)$ we assume that $\mu_a$ and $\mu_b$ are parameters to be estimated. In the setting of constant (but unknown) amplitude and baseline, one may use $(a_t, b_t)^T = (\mu_a,\mu_b)^T$, which simplifies the estimation significantly.

\subsection{Rao-Blackwellized Fixed-Lag Particle Smoothing} \label{sec:RBFLPS}
Statistical inference of the model introduced in Section~\ref{sec:PracticalModel} using an EM algorithm requires an approximation of the  joint smoothing distributions  $p(\mathbf{x}_{0:t}|y_{1:T})$ for $0\le t \le T$. In order to approximate these posterior distributions we use a particle smoother (cf. Doucet et.al., 2001). The aim of particle filters and smoothers introduced in Gordon et al. (1993) and Kitagawa (1996) is to obtain recursively an approximation of the posterior distributions of the states given the observations using importance sampling and importance resampling steps. These approximations are based on a set of points, the \textit{particles}, associated with importance weights. 

Assume that the set of weighted particles
$\{(\mathbf{x}_{0:t-1}^i,  \omega_{t-1}^i)\}_{i=1}^N$ approximates
$p(\mathbf{x}_{0:t-1}|y_{1:t-1})$ by $\sum_{i=1}^N \omega_{t-1}^i \,\delta_{\mathbf{x}_{0:t-1}^i}(\mathbf{x}_{0:t-1})$, with $\delta$ the Dirac delta function. Then, for all $1\le i \le N$,
\begin{enumerate}[-]
  \item Sample $\mathbf{x}_t^i \sim p(\mathbf{x}_t|\mathbf{x}_{t-1}^i)$.
  \item Compute the importance weight $\breve{\omega}_t^i \propto \omega_{t-1}^i
  p(y_t|\mathbf{x}_t^i)$.
  \end{enumerate}
%  \item For $i=1, \ldots, N$:
%\begin{itemize}
%  \item Normalize importance weights $\omega_t^i = \breve{\omega}_t^i/(
%\sum_{j=1}^N \breve{\omega}_t^j)$.
%\end{itemize}
%\end{itemize}
%
The new particles $\{(\mathbf{x}_{0:t}^i,  \omega_{t}^i)\}_{i=1}^N$, where $\{\omega_{t}^i\}_{i=1}^N$ are the normalized weights obtained from $\{\breve{\omega}_{t}^i\}_{i=1}^N$, approximate the posterior distribution $p(\mathbf{x}_{0:t}|y_{1:t})$. To avoid weight degeneracy, a resampling step that maps the
particle system $\{(\mathbf{x}_{0:t}^i, \omega_t^i)\}_{i=1}^N$ onto an
equally-weighted particle system is introduced, where each new particle is chosen in $\{\mathbf{x}_{0:t}^i\}_{i=1}^N$ according to $\{\omega_t^i\}_{i=1}^N$.  %$\{\mathbf{x}_{0:t}^i, 1/N \}_{i=1}^N$.
Resampling is carried out whenever the effective sample size (Kong et.al., 1994) $(\sum_{i=1}^N (\omega_t^i)^2)^{-1}$ is below some threshold. We favor the systematic resampling
with threshold $0.2N$ (see Douc et.al., 2005, for alternatives).

To avoid degeneracy of past samples as new observations become available, we propose to use the fixed lag smoother of Olsson et al. (2008). For computational reasons, this method is likely to be more efficient than algorithms based on forward backward decompositions as the `Forward Filtering Backward Smoothing algorithm' of Huerzeler and K\"{u}nsch (1998) and Doucet et.al. (2000a) or the `Forward Filtering Backward Simulation algorithm' of Godsill et.al. (2004). The basic idea of fixed-lag smoothing is that $p(\mathbf{x}_{0:t}| y_{1:T})$ should be close to $p(\mathbf{x}_{0:t}| y_{1:t(\ell)})$, with $t(\ell) \eqdef \min\{t+\ell,T\}$, for a well chosen lag $\ell$ (no resampling step is performed on the past samples when observations are obtained after time $t(\ell)$). As noted by Doucet et.al. (2000b),
the posterior distribution can be decomposed as
\begin{equation*}
p(\mathbf{x}_{0:t(\ell)}| y_{1:t(\ell)}) =  p(a_{0:t(\ell)}, b_{0:t(\ell)}|y_{1:t(\ell)}, \phi_{0:t(\ell)}) \, p(\phi_{0:t(\ell)}, \psi_{0:t(\ell)}| y_{1:t(\ell)})\,.
\end{equation*}
$p(a_{0:t(\ell)}, b_{0:t(\ell)}|y_{1:t(\ell)}, \phi_{0:t(\ell)})$ is then computed by a Kalman smoother
while $p(\phi_{0:t(\ell)}, \psi_{0:t(\ell)}| y_{1:t(\ell)})$ is approximated by samples
$\{(\phi_{0:t(\ell)}^i, \psi_{0:t(\ell)}^i)^T,  \omega_{t(\ell)}^i\}_{i=1}^N$ from the
particle smoother. Thus the above relation implies that the marginal densities
$p(a_{t}, b_{t},\phi_t|y_{1:T})$ are approximated by
\begin{equation} \label{SmootherDensityApproximation}
p(a_{t}, b_{t},\phi_t|y_{1:T}) \approx \sum_{i=1}^N \tilde{\omega}_{t}^i \,
\mathcal{N}\Big(a_t, b_t \big| (\tilde{a}_t^i, \tilde{b}_t^i)^T, \tilde{\Sigma}_t^i \Big) \,
\delta_{\phi_{t}^i}(\phi_{t})\,,
\end{equation}
where $\tilde{\omega}_{t}^i = \omega_{t(\ell)}^i$, $(\tilde{a}_t^i, \tilde{b}_t^i) := (a_{t|t(\ell)}^i,b_{t|t(\ell)}^i) = \mathbf{E} \big( (a_t,b_t) \big|y_{1:t(\ell)},\phi_{0:t(\ell)}^i\big)$
and $\tilde{\Sigma}_t^i := \Sigma_{t|t(\ell)}^i$ are computed
with the Kalman smoother. Smoothing by marginalization has been criticized for causing sample
impoverishment (Doucet et.al., 1999). While
this is true in general, it is not an issue in the setting of this article
because the lag $\ell$ is not large and the resampling frequency is  low.
In contrast to smoothing algorithms which proceed backwards in time
(c.f. Godsill et.al., 2004; Doucet et.al, 2000a), smoothing
by marginalization has the advantage that it can be applied on-line.
When the observation at time $t(\ell)$ comes in, the estimates of time $t$ can be updated using the
fixed-lag smoothing density. In addition, it is computationally very cheap.
The RBPS has computational costs $\mathcal{O}(\ell NT)$ for $T$ smoothing time
steps. At each iteration, only the particles for times $t(\ell)-1,\ldots, t$ are
required, implying a storage requirement of $\mathcal{O}(\ell N)$. The following algorithm
can e.g. be found in Shumway and Stoffer (2011), Property 6.2 and 6.3.

\bigskip

\hrule

\bigskip

\noindent
\textbf{Rao-Blackwellized Fixed-Lag Particle Smoother (RBPS)}

\bigskip

\hrule

\bigskip

\noindent
{\em Initialization} (for $t = 0$)
      \vspace*{-0.2cm}
	\begin{itemize}
       \itemsep-0.2cm
		\item {\textbf{For} $i=1,\ldots,N$}:  Sample $(\phi_0^i, \psi_0^i)^T \sim
		p(\phi_0, \psi_0)$, set $\omega_0^i = 1$, and choose $a_0^i$, $b_0^i$,
		$\Sigma_0^i$ according to prior knowledge.
		\end{itemize}
{\em Filtering} (for $t = 1,2,\ldots$)
      \vspace*{-0.2cm}
	\begin{enumerate}
      \itemsep-0.2cm
      \item {\em Kalman Prediction Step}
      \vspace*{-0.2cm}
      \begin{itemize}
        \item {\textbf{For} $i=1,\ldots,N$}: Compute with $\mu := (\mu_a,\mu_b)^{T}$
        \begin{equation*}
        (a_{t|t-1}^i, b_{t|t-1}^i)^T = \mu + A \big[(a_{t-1}^i, b_{t-1}^i)^T - \mu\big],\quad \Sigma_{t|t-1}^i = A \Sigma_{t-1}^i A^T + Q.
        \end{equation*}
      \end{itemize}
      \item {\em Importance Sampling Step}
            \vspace*{-0.2cm}
      \begin{itemize}
        \item {\textbf{For} $i=1,\ldots,N$}: Sample $(\phi_t^i, \psi_t^i)^T
        \sim p(\phi_t, \psi_t|\phi_{t-1}^i, \psi_{t-1}^i)$, compute\\
        $F_{t|t-1}^i=C_t^i \Sigma_{t|t-1}^i (C_t^i)^T + \sigma^2_{\varepsilon}$
        with $C_t^i = (f(\phi_t^i), 1)$ and evaluate importance weights
        \begin{equation*}
        \breve{\omega}_t^i \propto \omega_{t-1}^i \,p(y_t|y_{1:t-1},
        \phi_{0:t}^i) = \omega_{t-1}^i \, \mathcal{N}\Big(y_t\big|C_t^i
        (a_{t|t-1}^i, b_{t|t-1}^i)^T, F_{t|t-1}^i\Big).
        \end{equation*}
		\item {\textbf{For} $i=1,\ldots,N$}: Normalize importance weights
		$\omega_t^i  = \breve{\omega}_t^i/(\sum_{j=1}^N \breve{\omega}_t^j)$.
      \end{itemize}
      \item {\em Resampling Step if $(\sum_{i=1}^N (\omega_t^i)^2)^{-1}<0.2N$}
       \vspace*{0.2cm}
      \item {\em Kalman Updating Step}
            \vspace*{-0.2cm}
      \begin{itemize}
        \item {\textbf{For} $i=1,\ldots,N$}: Compute
        \begin{eqnarray*}
        %K_t^i &=& \Sigma_{t|t-1}^i (C_t^i)^T (F_{t|t-1}^i)^{-1},\\
        \big(a_t^i, b_t^i\big)^T &=& \big(a_{t|t-1}^i, b_{t|t-1}^i\big)^T + \Sigma_{t|t-1}^i
        (C_t^i)^T  \left\{y_t - C_t^i \big(a_{t|t-1}^i, b_{t|t-1}^i\big)^T
        \right\} (F_{t|t-1}^i)^{-1},\\
        %\Sigma_t^i &=& \{I-K_t^i C_t^i \} \Sigma_{t|t-1}^i.
        \Sigma_t^i &=& \Sigma_{t|t-1}^i - \left\{\Sigma_{t|t-1}^i (C_t^i)^T
        C_t^i \Sigma_{t|t-1}^i \right\} (F_{t|t-1}^i)^{-1}.
        \end{eqnarray*}
      \end{itemize}
      %\noindent As above let $(a_{k}^i, b_{k}^i)=(a_{k|k}^i, b_{k|k}^i)$ and $(\tilde{a}_{k}^i,\tilde{b}_{k}^i)=(a_{k|t+l}^i, b_{k|t+l}^i)\,$ (where $t+l$ is fixed).
      \end{enumerate}
      {\em Smoothing}
      \begin{enumerate}
            \item[5.] {\em Kalman Smoothing Step} (for $k = t(\ell)-1,\ldots,t$)
            \vspace*{-0.2cm}
      \begin{itemize}
        \item {\textbf{For} $i=1,\ldots,N$}: Compute
        \begin{eqnarray*}
        %\Sigma_{k+1|k}^i &=& A \Sigma_k^i A^T + Q,\\
        V_k^i &=& \Sigma_k^i A^T (\Sigma_{k+1|k}^i)^{-1},\\
        (\tilde{a}_k^i, \tilde{b}_k^i)^T &=& (a_k^i, b_k^i)^T + V_k^i \left\{
        \big(\tilde{a}_{k+1}^i, \tilde{b}_{k+1}^i\big)^T - \big(a_{k+1|k}^i,
        b_{k+1|k}^i\big)^T\right\}, \\
        \tilde{\Sigma}_k^i &=& \Sigma_k^i + V_k^i (\tilde{\Sigma}_{k+1}^i -
        \Sigma_{k+1|k}^i) (V_k^i)^T, \\
        \tilde{\Sigma}_{k,k-1}^i &=& \Sigma_k^i \big(V_{k-1}^i)^T +
        V_k^i(\tilde{\Sigma}_{k+1,k}^i - A \Sigma_{k}^i\big) (V_{k-1}^i)^T,
        \end{eqnarray*}
      \end{itemize}
      with initial values
      \[
      (\tilde{a}_{t(\ell)}^i, \tilde{b}_{t(\ell)}^i)^T =
      (a_{t(\ell)}^i, b_{t(\ell)}^i)^T\,,\;\tilde{\Sigma}_{t(\ell)}^i = \Sigma_{t(\ell)}^i \quad \mbox{and}\quad
      \tilde{\Sigma}_{t(\ell),t(\ell)-1}^i = \big(I-K_{t(\ell)}^i C_{t+\ell}^i\big) A \,\Sigma_{t(\ell)-1}^i\,,
      \]
      where $K_{t(\ell)}^i:= \Sigma_{t(\ell)|t(\ell)-1}^i (C_{t(\ell)}^i)^T (F_{t(\ell)|t(\ell)-1}^i)^{-1}$ is the Kalman gain.\\
      Furthermore $\big(a_{k+1|k}^i,
        b_{k+1|k}^i\big)^T = A (a_k^i,
        b_k^i)^T + (I-A)\mu$.
      \item[6.] {\em Result}
            \vspace*{-0.2cm}
      \begin{itemize}
        \item Obtain amplitude estimate $\hat{a}_{k} =
      \sum_{i=1}^N \omega_{t(\ell)}^i \, \tilde{a}_{k}^i$, baseline estimate $\hat{b}_{k} =
      \sum_{i=1}^N \omega_{t(\ell)}^i \, \tilde{b}_{k}^i$, and phase estimate
      $\hat{\phi}_{k} = \sum_{i=1}^N \omega_{t(\ell)}^i \, \phi_{k}^i$ for time $k = t$.
      \end{itemize}
	\end{enumerate}
\hrule

\subsection{The EM Algorithm for the parametric MLE} \label{sec:EM_Algorithm}

In this section, we assume that $f$ is known and estimate $\theta = (\alpha, \beta, \sigma^2_{\varepsilon}, \mu, A, Q)$ based on a stochastic EM algorithm (Dempster et.al., 1977). Shumway and Stoffer (1982) had introduced the EM algorithm for linear Gaussian state space models.
Wei and Tanner (1990) (see also Tanner 1993) had proposed to replace the E-step by
a Monte Carlo integration leading to the MCEM Algorithm.
In the present model the equations for the
`Gaussian part' take the same form (conditional on $\phi_t$) as the original equations
leading to a partial MCEM Algorithm. This reduces the computational complexity considerably.

Assume that the signal
$y_t$ is received up to time $T$. The EM algorithm maximizes the likelihood $p_{\theta}(y_{1:T})$ iteratively. In the E-step, the $\mathcal{Q}(\theta|\theta^{(m)}) = \mathbf{E}_{\theta^{(m)}} [\log
p_{\theta}(\mathbf{x}_{0:T}, y_{1:T}) |y_{1:T}]$ is approximated, where $\theta^{(m)}$ is the current estimate.
We have
\begin{align} \label{Q-function}
\mathcal{Q}(\theta|\theta^{(m)})
& =  \mathbf{E}_{\theta^{(m)}} [\log
p(\phi_{0}, \psi_{0}) |y_{1:T}] + \sum_{t=1}^T
\mathbf{E}_{\theta^{(m)}} [\log p_{\theta}(y_{t}| \mathbf{x}_{t}) |y_{1:T}]\\
& \;\;
+ \sum_{t=1}^T \mathbf{E}_{\theta^{(m)}} [\log
p_{\theta}(a_{t}, b_{t}| a_{t-1}, b_{t-1}) |y_{1:T}] +
\sum_{t=1}^T \mathbf{E}_{\theta^{(m)}} [\log p_{\theta}(\phi_{t}, \psi_{t}|
\phi_{t-1}, \psi_{t-1}) |y_{1:T}]\,. \nonumber
\end{align}
It follows that $\mathcal{Q}(\theta|\theta^{(m)})$ could be approximated through
smoothing particles, which are generated with respect to the parameter value
$\theta^{(m)}$. Due to the computational complexity we use in this paper only the fixed lag smoother,
i.e. we replace  $\mathbf{E}_{\theta^{(m)}} [\,\cdot\, |y_{1:T}]$
by $\mathbf{E}_{\theta^{(m)}} [\,\cdot\, |y_{1:t(\ell)}]$ which can be calculated by the RBPS. The difference should be minor for reasonable $\ell$ (see Olsson et al. (2008) for an explicit control of the $\mathrm{L}_p$-mean error of the fixed lag smoother when applied to additive functionals). With
\[
\tilde{S}_t^i := \mean \big[(a_t, b_t)^T(a_t, b_t)\big| y_{1:t(\ell)},\phi_{0:t(\ell)}^i\big] = \tilde{\Sigma}_t^i + (\tilde{a}_t^i, \tilde{b}_t^i)^T
(\tilde{a}_t^i, \tilde{b}_t^i)\quad \mbox{and}\quad C_t^i := (f(\phi_t^i), 1)\,,
\]
we obtain
\begin{align*}
\hat{\mathcal{Q}}(\theta|\theta^{(m)})
& =  \text{const} - \frac12 \sum_{t=1}^{T} \sum_{i=1}^N
\tilde{\omega}_{t}^i \left[ \log 2\pi + \log \sigma_{\varepsilon \star}^2 +
\frac{1}{\sigma_{\varepsilon \star}^2} \left\{ y_t^2 - 2 C_t^i (\tilde{a}_t^i,
\tilde{b}_t^i)^T y_t + C_t^i \tilde{S}_t^i (C_t^i)^T \right\} \right]\\
& \quad  - \frac12 \sum_{t=1}^{T} \sum_{i=1}^N
\tilde{\omega}_{t}^i \bigg[ 2 \log 2\pi + \log |Q| + \text{tr} \bigg\{ Q^{-1} \Big(\tilde{\Sigma}_t^i + \big((\tilde{a}_t^i, \tilde{b}_t^i)^T - \mu \big) \big((\tilde{a}_t^i, \tilde{b}_t^i) - \mu^{T} \big)\Big)\\
& \hspace*{3.4cm} - Q^{-1} A \Big(\tilde{\Sigma}_{t-1,t}^i + \big((\tilde{a}_{t-1}^i, \tilde{b}_{t-1}^i)^T - \mu \big) \big((\tilde{a}_{t}^i, \tilde{b}_{t}^i) - \mu^{T} \big)\Big)\\
& \hspace*{3.4cm} - Q^{-1} \Big(\tilde{\Sigma}_{t,t-1}^i + \big((\tilde{a}_{t}^i, \tilde{b}_{t}^i)^T - \mu \big) \big((\tilde{a}_{t-1}^i, \tilde{b}_{t-1}^i) - \mu^{T} \big)\Big)A^{T}\\
& \hspace*{3.4cm} + Q^{-1} A \Big(\tilde{\Sigma}_{t-1}^i + \big((\tilde{a}_{t-1}^i, \tilde{b}_{t-1}^i)^T - \mu \big) \big((\tilde{a}_{t-1}^i, \tilde{b}_{t-1}^i) - \mu^{T} \big)\Big)A^{T} \bigg\} \bigg]
\end{align*}
\vspace*{-0.5cm}
\begin{equation*}
  + \sum_{t=1}^T \sum_{i=1}^N
\tilde{\omega}_{t}^i \log p_{\alpha, \beta}(\phi_{t}^i, \psi_{t}^i|
\phi_{t-1}^i, \psi_{t-1}^i)\,. \qquad \qquad \qquad \qquad \qquad \qquad
\end{equation*}

\setlength{\textheight}{225mm}%{240mm}
\setlength{\textwidth}{6.1in}%{5.9in}
\setlength{\topmargin}{-11mm}% for pdf

Maximization with
respect to $\sigma^2_{\varepsilon}$, $\mu$, $A$, and $Q$ yields in the M-step the estimates
\begin{eqnarray*}
(\sigma_{\varepsilon }^2)^{(m+1)} &=& \frac1T \sum_{t=1}^{T} \sum_{i=1}^N
\tilde{\omega}_{t}^i \left\{ y_t^2 - 2 C_t^i (\tilde{a}_t^i,
\tilde{b}_t^i)^T y_t + C_t^i \tilde{S}_t^i (C_t^i)^T \right\},\\
\mu^{(m+1)} &=& \frac1T\sum_{t=1}^{T} \sum_{i=1}^N \tilde{\omega}_{t}^i \, (\tilde{a}_t^i,
\tilde{b}_t^i)^T + O_{p}\big(\frac {1} {T}\big),\\
A^{(m+1)} &=& \,\left( \sum_{t=1}^{T} \sum_{i=1}^N \tilde{\omega}_{t}^i
\, \Big(\tilde{\Sigma}_{t,t-1}^i + \tilde{D}_{t,t-1}^i\Big)\right) \left( \sum_{t=1}^{T} \sum_{i=1}^N \tilde{\omega}_{t}^i
\, \Big(\tilde{\Sigma}_{t-1}^i + \tilde{D}_{t-1}^i\Big) \right)^{-1},\\
Q^{(m+1)} &=& \frac1T \left\{ \sum_{t=1}^{T} \sum_{i=1}^N
\tilde{\omega}_{t}^i \, \Big(\tilde{\Sigma}_{t}^i + \tilde{D}_{t}^i\Big) - A^{(m+1)} \sum_{t=1}^{T}
\sum_{i=1}^N \tilde{\omega}_{t}^i \, \Big(\tilde{\Sigma}_{t-1,t}^i + \tilde{D}_{t-1,t}^i\Big)\right\},
\end{eqnarray*}
with
\begin{eqnarray*} \label{}
\tilde{D}_t^i &=&  \big((\tilde{a}_t^i,
\tilde{b}_t^i)^T - \mu^{(m+1)}\big) \big((\tilde{a}_t^i,
\tilde{b}_t^i)^{T} - \mu^{(m+1)}\big)^{T},\\
\tilde{D}_{t,t-1}^i &=&  \big((\tilde{a}_t^i,
\tilde{b}_t^i)^T - \mu^{(m+1)}\big) \big((\tilde{a}_{t-1}^i,
\tilde{b}_{t-1}^i)^{T} - \mu^{(m+1)}\big)^{T}. \qquad \qquad \qquad
\end{eqnarray*}
In case where $\mu$ is assumed to be known we set $\mu^{(m+1)}=\mu$ (e.g. when $f$ is a nonparametric function $\mu_a$ and $\mu_b$ are usually set to $1$ and $0$ respectively). For $\alpha$ and $\beta$, numerical maximization is required because no
closed-form expression can be derived.

%\newpage
\bigskip

\hrule

\medskip

\noindent
%\textbf{EM Step}

%\vskip 2mm

\noindent
	\begin{enumerate}
      \item[7.] {\em (Parametric) EM Step} (see also 8. below)
      \vspace*{-0.2cm}
      \begin{itemize}
        \item Compute $(\sigma_{\varepsilon}^2)^{(m+1)}$,
        $\mu^{(m+1)}$,$A^{(m+1)}$, and $Q^{(m+1)}$.
        \item Obtain $\alpha^{(m+1)}$ and $\beta^{(m+1)}$ by numerical maximization of $\hat{\mathcal{Q}}_t(\alpha,
        \beta|\alpha^{(m)}, \beta^{(m)})$ to .
        \end{itemize}
    \end{enumerate}
\hrule

\medskip

\subsection{Nonparametric Estimation of the Oscillation Pattern} \label{sec:NonparametricEstimate}

If $f$ is unknown we have to maximize in addition
the second term in (\ref{Q-function}) with respect to $f$. The other estimates
remain unchanged including $(\sigma_{\varepsilon}^2)^{(m+1)}$ where $C_t^i $ must be replaced
by $\widehat{C}_t^i := (f^{(m+1)}(\phi_t^i), 1)$). For simplicity we ignore  the other terms, i.e. we maximize
\begin{equation}
\label{NonparaEM}
\mathcal{Q}(f|f^{(m)})  \propto  \text{const } - \sum_{t=1}^T \mathbf{E}_{f^{(m)}} [\{y_t -
a_t f(\phi_t) - b_t\}^2 |y_{1:t(\ell)}]
\end{equation}
with respect to $f$. As for nonparametric maximum likelihood estimation we need some
regularization in order to obtain a proper estimator. The basic idea for regularization
now is to approximate the densities $p(a_{t}, b_{t}, \phi_{t}| y_{1:t})$ in (\ref{NonparaEM})
instead of (\ref{SmootherDensityApproximation}) by the kernel density
\begin{equation} \label{kernelprobdistribution}
p(a_{t}, b_{t}, \phi_{t}| y_{1:t(\ell)})
\approx \sum_{i=1}^N \tilde{\omega}_{t}^i \,
\mathcal{N}\Big(a_t, b_t \big| (\tilde{a}_t^i, \tilde{b}_t^i)^T, \tilde{\Sigma}_t^i \Big)\, K_{h}\big(\phi_{t}-\phi_{t}^i\big)\,,
\end{equation}
where $K_{h}(\cdot) := \frac {1} {h} K\big(\frac {\cdot} {h}\big)$ with a bandwidth $h$ and a kernel $K (\cdot)$. Without prior knowledge both are equally good.
The following proposition shows that this leads to an estimator for $f$ which
is also based on kernel approximations. $h$ might be chosen adaptively (say as in Lepski et.al., 1997) but the situation is different here since $N$ can be chosen arbitrarily large - leading to $NT$ data points thus allowing for choosing $h$ arbitrarily small. In any case the optimal $h$ depends on the unknown $f$ and the number of data meaning that $h$ should be the same in each iteration step.
\begin{proposition}\label{NEM:prop1}
Suppose that the density $p(a_{t}, b_{t}, \phi_{t}| y_{1:t})$ is as in (\ref{kernelprobdistribution}). Then ${\mathcal{Q}}(f|f^{(m)})$
is maximized by the estimate
\begin{equation} \label{NEM:optEstimator2}
\tilde{f}^{(m+1)}(\phi) =
\frac{
\sum_{t=1}^T \sum_{i=1}^N
\tilde{\omega}_t^i \, K_{h}\big( (\phi-\phi_{t}^i) \mbox{ mod } 2\pi \big) \big\{ y_t
\tilde{a}_t^i - (\tilde{S}_t^i)_{12} \big\}
}{
\sum_{t=1}^T
\sum_{i=1}^N
\tilde{\omega}_t^i \,K_{h}\big( (\phi-\phi_{t}^i) \mbox{ mod } 2\pi \big)
(\tilde{S}_t^i)_{11} } \,.
\end{equation}
The analogue result holds when using, instead of the fixed lag smoother, the filter
(where $\tilde{a}_t^i$ and $\tilde{S}_t^i$ are replaced by $a_t^i$ and $S_t^i$ respectively)
or the complete smoother.
\end{proposition}

\proof See the appendix.

\bigskip

Each step of the above nonparametric EM algorithm improves the likelihood in ``nearly all cases''.
The latter restriction comes from the approximation (\ref{kernelprobdistribution}) which needs to be ``good enough''. More precisely we have with Jensen's inequality
\begin{align*}
\log \frac{p_{f^{(m+1)}}(y_{1:T})} {p_{f^{(m)}}(y_{1:T})}
%& =  \log \mathbf{E}_{f^{(m)}} \!\left[ \frac{p_{f^{(m+1)}}(\mathbf{x}_{0:T},
%y_{1:T})} {p_{f^{(m)}}(\mathbf{x}_{0:T}, y_{1:T})} \bigg| y_{1:T} \right] \\
%%
%& \geq  \mathbf{E}_{f^{(m)}} \!\left[ \sum_{t=1}^T \log
%\frac{p_{f^{(m+1)}}(y_{t}| \mathbf{x}_{t})} {p_{f^{(m)}}(y_{t}|
%\mathbf{x}_{t},)} \bigg| y_{1:T} \right] \\
%& =
\ge \sum_{t=1}^T \mathbf{E}_{f^{(m)}} \big[\{y_t \!-\!
a_t f^{(m)}(\phi_t) \!-\! b_t\}^2 |y_{1:T}\big] \!-\!
\sum_{t=1}^T \mathbf{E}_{f^{(m)}} \big[\{y_t \!-\!
a_t f^{(m+1)}(\phi_t) \!-\! b_t\}^2 |y_{1:T}\big].
\end{align*}
If $\tilde{f}^{(m+1)}$ would maximize (\ref{}) we had  $p_{f^{(m+1)}}(y_{1:T}) \geq p_{f^{(m)}}(y_{1:T})$.
Since we have used the approximation (\ref{kernelprobdistribution}) this is however not guaranteed in a strict sense.

\bigskip

\noindent \underline{Improving the speed of convergence of the MSEM - algorithm}:\\[6pt]
An example for the iteration steps of the nonparametric EM-estimate is given in Figure~\ref{fig:ecgObsFuncEstC} (for details see Section~\ref{sec:HumanECG}). If no prior
information on the shape of $f$ is used (when choosing $\hat{f}^{(0)}$) then convergence may be slow (since a poor $\hat{f}^{(0)}$ will lead to poor particles
and then again to a poor $\hat{f}^{(1)}$). To speed up convergence we may invoke two kinds of additional information into the algorithm that speed up convergence
considerably. The idea is to use these corrections only during the first few steps:

\medskip

\noindent 1) The first correction uses that the empirical distribution of the points $\phi_t$ mod $2\pi$ converges to an uniform distribution.
This follows since the increments $\Delta \phi_t$ are supposed to be stationary. By using this property we can exclude time warping (see \ref{enumerate1:timewarping}) in Section~\ref{sec:Identifiability}) and the correction
consists of transforming the phases accordingly with the edf of the $\phi_t$ mod $2\pi$. Let $\hat{F}_{\phi}$ be a smoothed version of this edf given by
\begin{equation}\label{phasedistrfunc:kernelest}
\hat{F}_{\phi}(y) = \int_0^y \hat{p}_{\phi}(\phi) \, \rmd \phi,
\end{equation}
where
\begin{equation}\label{foldedphase:kernelest}
\hat{p}_{\phi}(\phi) =  \sum_{t=1}^T \sum_{i=1}^N \tilde{\omega}_t^i \,
K_{h_{\phi}}\!\big((\phi-\phi_t^i) \, \text{mod}\, 2\pi \big)
\end{equation}
(note that it is computationally more efficient to use a
frequency polygon instead of
(\ref{foldedphase:kernelest}) since then the distribution
function and the inverse distribution can be easily calculated).

We then remove possible time warping by transforming $f$ and the particles to
\begin{equation*} \label{}
\check{f}^{(m+1)}(\check{\phi}) = \hat{f}^{(m+1)} \Big(\hat{F}_{\phi}^{-1}\Big(\frac {\check{\phi} \, \text{mod} \,2\pi} {2\pi}\Big)\!\Big) \quad \mbox{and} \quad \check{\phi}_t^i = 2\pi \big\{ \hat{F}_{\phi}(\phi_t^i\, \text{mod} \, 2\pi) + \lfloor
\phi_t^i / (2 \pi) \rfloor \big\}.
\end{equation*}
\noindent (for simplicity we denote in step 2) the new $\check{f}^{(m+1)}$ and $\check{\phi}_t^i$ again by $\hat{f}^{(m+1)}$ and $\phi_t^{i}$).
\medskip

\noindent 2) The second correction uses the information that all $2\pi$-periodic behavior of the signal is due to the periodicity of $f$ and not to any
periodic behavior of the amplitude $a_t$ and the baseline $b_t$ (this follows from the independence of the process $\phi_t$ from $a_t$ and $b_t$). Using this
information means to remove all $2\pi$-periodic structures from the amplitude and baseline estimates and to transfer them to the oscillation pattern $f$, i.e.
to make the transformation
\begin{align*} \label{}
\check{f}^{(m+1)}(\phi) &= \hat{a}_{\text{per}}(\phi)\times \hat{f}^{(m+1)}(\phi) + \hat{b}_{\text{per}}(\phi);\\
\check{a}_t^i &= \tilde{a}_t^i / \hat{a}_{\text{per}}(\phi_t^i \, \text{mod} \, 2\pi); \qquad \check{b}_t^i = \tilde{b}_t^i - \hat{b}_{\text{per}}(\phi_t^i \, \text{mod} \, 2\pi).
\end{align*}
Here
\begin{equation*} \label{}
\hat{a}_{\text{per}}(\phi) = \frac{\sum_{t=1}^T \sum_{i=1}^N \tilde{\omega}_t^i \, K_{h_a}\!\big((\phi-\phi_t^i) \,  \text{mod} \, 2\pi\big) \, \tilde{a}_t^i / \bar{a} (\phi_t^i)} {\sum_{t=1}^T
\sum_{i=1}^N \tilde{\omega}_t^i \, K_{h_a}\!\big((\phi-\phi_t^i) \,  \text{mod} \, 2\pi\big)} , \quad \phi \in (0,2\pi], \quad \,
\end{equation*}
where
\begin{equation*} \label{}
\bar{a} (\phi) = \frac{\sum_{t=1}^T \sum_{i=1}^N \tilde{\omega}_t^i \,I_{\{-2\pi < \phi_t^i-\phi \le 2\pi\}} \,
\tilde{a}_t^i} {\sum_{t=1}^T \sum_{i=1}^N \tilde{\omega}_t^i \, I_{\{-2\pi < \phi_t^i-\phi \le 2\pi\}}}, \quad  \phi \in (0,\infty), \quad \qquad \qquad \qquad
\end{equation*}
smoothes the amplitude and the baseline by an average over 2 periods (the necessity of $\bar{a} (\phi)$ becomes obvious when considering the example
where $\tilde{a}_t^i$ is linear). Similarly
\begin{equation*} \label{}
\hat{b}_{\text{per}}(\phi) = \frac{\sum_{t=1}^T \sum_{i=1}^N \tilde{\omega}_t^i \, K_{h_b} \big((\phi-\phi_t^i) \,  \text{mod} \, 2\pi\big) \, \big(\tilde{b}_t^i - \bar{b} (\phi_t^i)\big)} {\sum_{t=1}^T
\sum_{i=1}^N \tilde{\omega}_t^i \, K_{h_b}\!\big((\phi-\phi_t^i) \,  \text{mod} \, 2\pi\big) } , \quad \phi \in (0,2\pi],
\end{equation*}
with
\begin{equation*} \label{}
\bar{b} (\phi) = \frac{\sum_{t=1}^T \sum_{i=1}^N \tilde{\omega}_t^i \,I_{\{-2\pi < \phi_t^i-\phi \le 2\pi\}} \,
\tilde{b}_t^i} {\sum_{t=1}^T \sum_{i=1}^N \tilde{\omega}_t^i \,I_{\{-2\pi < \phi_t^i-\phi \le 2\pi\}}},  \quad  \phi \in (0,\infty). \quad \qquad \qquad \qquad
\end{equation*}

\medskip

Performing first step 1 and then  2 has the disadvantage that all $\tilde{a}_t^i$ have to be recalculated after step 1 which is time consuming.
We therefore have combined both steps by
\begin{equation}
\label{nonpEM:transformation2}
 \check{f}^{(m+1)}(\breve{\phi}) = \hat{a}_{\text{per}}\Big(\hat{F}_{\phi}^{-1}\!\Big(\frac {\breve{\phi} \, \text{mod} \,2\pi} {2\pi}\Big)\!\Big) \! \times \!
 \hat{f}^{(m+1)} \Big(\hat{F}_{\phi}^{-1}\!\Big(\frac {\breve{\phi} \, \text{mod} \,2\pi} {2\pi}\Big)\!\Big)  +
 \hat{b}_{\text{per}}\Big(\hat{F}_{\phi}^{-1}\!\Big(\frac {\breve{\phi} \, \text{mod} \,2\pi} {2\pi}\Big)\!\Big)
\end{equation}

\medskip

\hrule

\medskip

\noindent
	\begin{enumerate}
      \item[8.] {\em Nonparametric EM Step}
            \vspace*{-0.2cm}
      \begin{itemize}
        \item Minimize the third and fourth term of (\ref{Q-function}) (with $y_{1:T}$ replaced by $y_{1:t+l}$)
        as described above to get estimates  $\mu^{(m+1)}$,$A^{(m+1)}$, $Q^{(m+1)}$, $\alpha^{(m+1)}$ and $\beta^{(m+1)}$,
        and the second term of (\ref{Q-function}) to get estimates $\tilde{f}^{(m+1)}$ (defined by (\ref{NEM:optEstimator2})) and $(\sigma_{\varepsilon}^2)^{(m+1)}$.
        \item Perform the correction steps 1) and 2) by using (\ref{nonpEM:transformation2}). The resulting estimator $\hat{f}^{(m+1)}$ is used in the next iteration.
        Iterate this until convergence. In the final step set $\hat{f}^{(m+1)}=\tilde{f}^{(m+1)}$ without using the correction.
        \item Keep the initial particles and weights $\phi_0^i$,
        $\tilde{a}_0^i$, $\tilde{b}_0^i$, and $\tilde{\omega}_0^i$, $i=1,
        \ldots, N$, for the next iteration (they have been updated due to the fixed lag smoother).
        \end{itemize}
    \end{enumerate}
\hrule

\bigskip

To start the iteration an initial guess $\hat{f}^{(0)}$ is required where prior information may be used. In
Section~\ref{sec:HumanECG} on ECG recordings it turned out that the
uninformative function $\hat{f}^{(0)}\equiv 0$ may suffice. However, one should choose the initial values $\alpha^{(0)}$ and $\beta^{(0)}$
such that the theoretical phase increment $\mathbf{E}[\Delta{\phi_{t}}] = \alpha^{(0)} / (1- \beta^{(0)})$  is close to the empirical one.
This can be achieved by counting the number of cycles in the data as in Section~\ref{sec:HumanECG} (see also (\ref{acd:mean})).

\section{Simulations and Data Examples} \label{sec:SimulationsExamples}

In this section we apply the proposed algorithms to simulated data including a noisy R\"{o}ssler attractor and to human electrocardiogram recordings.

\subsection{Simulated Data} \label{sec:SimulatedData}

\begin{figure} 
\centering
\includegraphics[width=380pt, height=380pt]{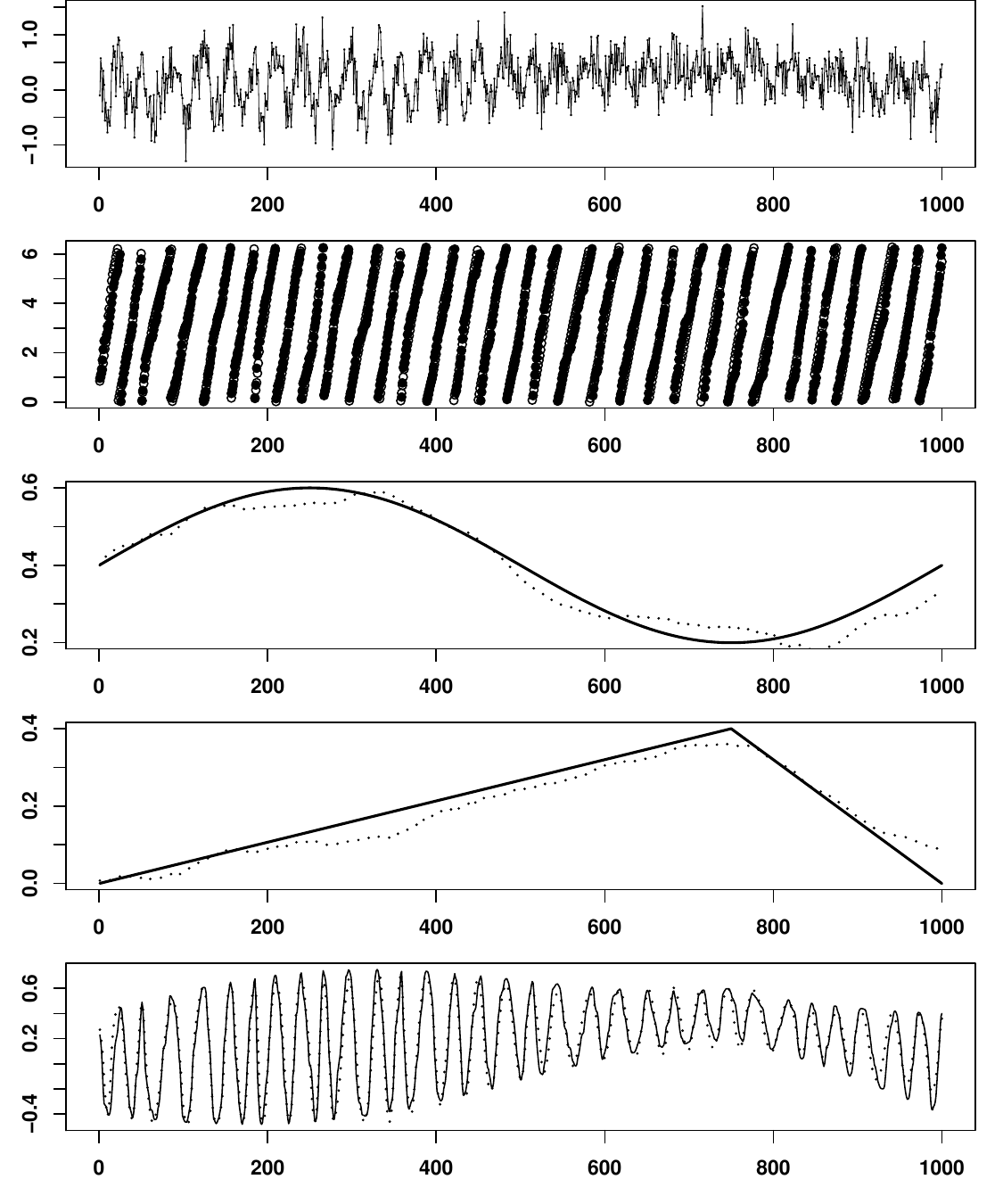}
\caption{\footnotesize
The estimation results of the RBPS for the simulated signal with
$\mathcal{N}(0, 0.16)$ noise (from top to bottom): The simulated noisy
observations; the folded estimated phase (circles) and the folded true phase
(solid circles); the estimated amplitude (dotted line) and the true amplitude
(solid line); the estimated baseline (dotted line) and the true baseline (solid
line); the simulated non-noisy signal (solid line) and the denoised signal obtained from the RBPS estimates
(dotted line).
}
\label{fig:EstimatesHighNoise}
\end{figure}

We generate observations $y_t$, $t=1,\ldots,1000$, from the GSSM
defined through (\ref{intro:acdstateequation}) and
\[
y_t = a_t \cos(\phi_t) + b_t + \varepsilon_t,
\]
with true $a_t = 0.2 \sin(2\pi t/1000) + 0.4$ and $b_t = 0.4t/750 \
\mathbf{1}_{t\leq 750} + (0.4 - 0.4 (t-750)/250) \ \mathbf{1}_{t > 750}$.
The ACD model parameters are set to $\alpha=0.2$ and $\beta = 0.01$. Two
levels of observation noise are investigated: $\sigma^2_{\varepsilon} = 0.01$
and $\sigma^2_{\varepsilon} = 0.16$. The parameters $(\alpha, \beta, \sigma^2_{\varepsilon}, \text{vec}(Q))$
are estimated with the (parametric) EM algorithm and we set
$A=\text{diag}(1,1)$. For both noise levels, the EM algorithm obtained
estimates $(\hat{\alpha}, \hat{\beta}, \hat{\sigma}^2_{\varepsilon})$ which were
very close to the true values after a few iterations. For $Q$, we obtained
$\text{diag}(10^{-4}, 5\times 10^{-5})$.

Figures \ref{fig:EstimatesHighNoise} shows
the true values and the estimated values for the high noise level (the low noise level even looks better).
The estimates of the amplitude, baseline, and (folded) phase are computed by the
RBPS with $N=500$ particles and lag $l=100$.
In addition, the figures display a signal reconstruction based on the estimates
(bottom plot) consisting of the estimated denoised observations $\hat{y}_t =
\hat{a}_t \cos{\hat{\phi}_t} + \hat{b}_t$. For comparison, also the non-noisy
observations $y_t - \varepsilon_t$ are given.
In the low-noise setting, the estimates are even more
accurate.

\subsection{The Noisy R\"{o}ssler Attractor} \label{sec:NoisyRoessler}

\begin{figure}
\centering
\includegraphics[width=380pt, height=350pt]{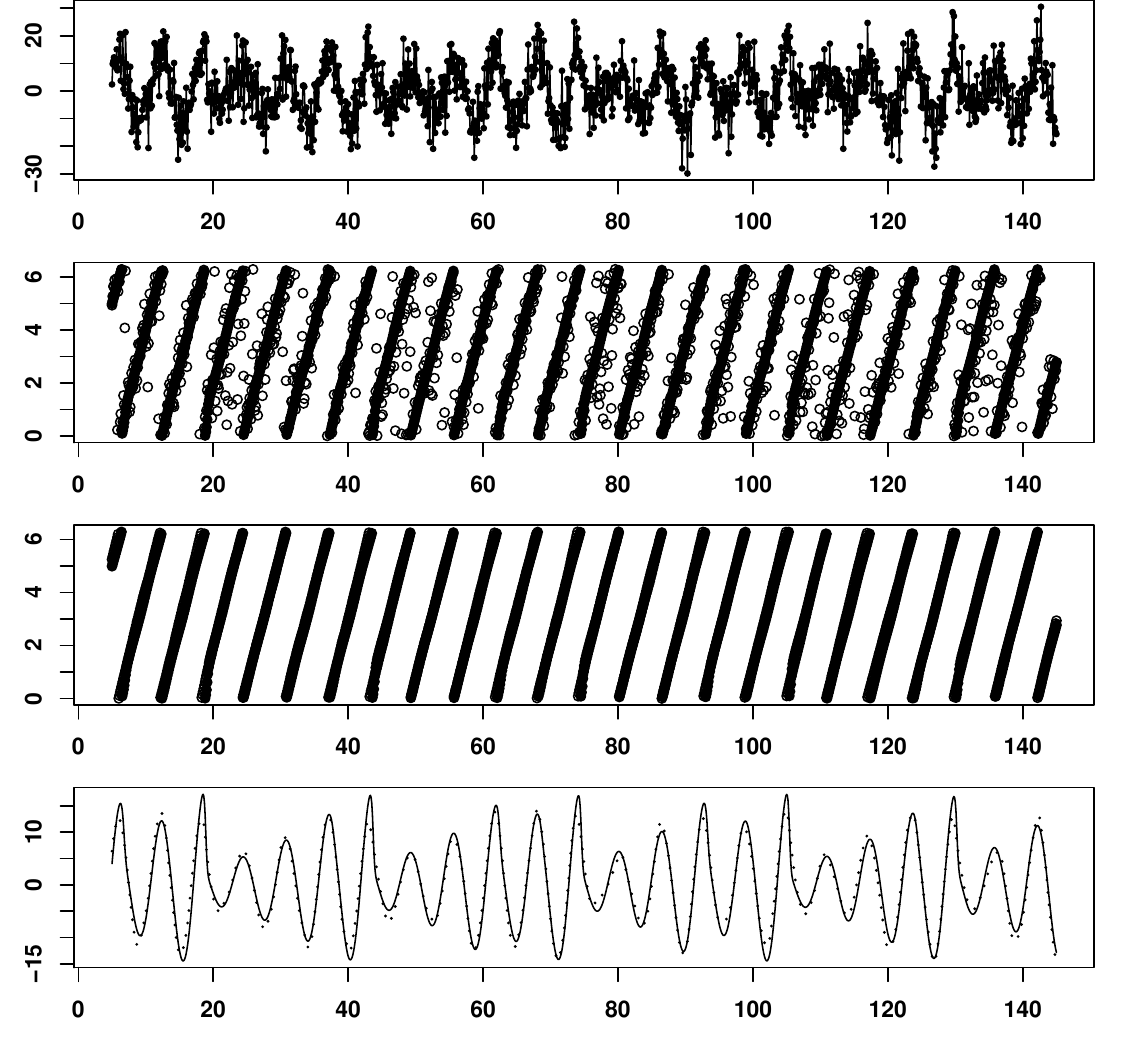}
\caption{\footnotesize Estimation results for the noisy R\"{o}ssler attractor. The
plots show (from top to bottom): $x_1$-component of the R\"{o}ssler attractor with
additive i.i.d. $\mathcal{N}(0,40)$ noise; the folded Hilbert phase (circles)
compared with the true folded phase (solid circles); the folded phase estimated
with the RBPS (circles) compared with the true folded phase (solid circles);
the (non-noisy) $x_1$-component of the R\"{o}ssler attractor compared with the
reconstructed (denoised) signal based on the amplitude and phase estimates of
the RBPS.}
\label{fig:RoesslerResultsVar40}
\end{figure}

We now consider the R\"{o}ssler attractor with  configuration
\begin{equation*}
\dot{x_1} = -x_2-x_3, \qquad
\dot{x_2} = x_1 + .15 x_2, \qquad
\dot{x_3} = .4 + x_3 ( x_1 - 8.5).
\end{equation*}
The R\"{o}ssler attractor and related systems are, for instance, used to model
population dynamics (Blasius et.al., 1999; Lloyd and May, 1999). We
focus on the $x_1$ component for which the (folded) phase can be defined by means of $\arctan(x_{2,t}/x_{1,t})$ (cf. Pikovsky et al., 1997).
It is assumed that $x_{1,t}$ is not observed directly but through $y_t = x_{1,t}
+ \varepsilon_t$. A standard method for estimating the phase is to apply the Hilbert transform  (cf. Rosenblum et. al., 1996).
%We use this   as a reference method.

We now use the cosine model
\begin{equation} \label{CosineModelRoessler}
y_t = a_t \cos(\phi_t) + \varepsilon_t
\end{equation}
in combination with the RBPS and the EM algorithm for estimation (i.e $b_t \equiv 0$).

We integrate the R\"{o}ssler system with step size $0.1$ using the Runge-Kutta
method (Press et al. 1992, pp. 710-714) and we add i.i.d. Gaussian noise
to the $x_1$-component. Again, two noise levels have been considered:
$\mathcal{N}(0,4)$ and $\mathcal{N}(0,40)$ (but we only display the high noise level - see Figure~\ref{fig:RoesslerResultsVar40}).
As parameter estimates we obtain $(\hat{\alpha}, \hat{\beta})^T = (0.2,
0.02)^T$, $\hat{Q} = \text{diag}(0.9, 0)$ (the second value is set to zero),
and $\hat{\sigma}^2_{\varepsilon}$ close to the true value. $A$ was set to $\text{diag}(1,0)$.
The RBPS is applied with $N=1000$ particles and lag $l=200$.
For the computation of the Hilbert phase a running window of 100 data points
is used. The (folded) phase estimates of the Hilbert transform and our method
together with the true phase are presented in the second and third plot of Figures~\ref{fig:RoesslerResultsVar40}
 (in the third plot the true phases and the estimates are almost identical). It can be observed, that the phase estimate of
the RBPS is much closer to the true phase than the Hilbert phase. The bottom
plot shows the (non-noisy) $x_1$-component of the R\"{o}ssler attractor along
with the denoised signal $\hat{y} = \hat{a}_t \cos(\hat{\phi}_t)$, where
$\hat{a}_t$ and $\hat{\phi}_t$ are obtained from the RBPS. Note, that even in
the high noise case, the denoised signal is very close to the true signal. The results in the
low noise case are even better.

\subsection{Application to Human Electrocardiogram Recordings} \label{sec:HumanECG}

Human ECG recordings are characterized by a specific oscillation pattern $f$,
amplitude changes, and baseline shifts. The oscillation pattern depends on certain characteristics of the specific human being. The baseline shifts are typically caused by
respiration or body movements (Clifford et.al. 2006). The estimation of the oscillation pattern $f$ is for example of importance for the diagnosis of several heart diseases - see Figure 3.6. and Table 3.1 in Clifford et.al. (2006). For inference we use the model
\begin{equation*} \label{}
y_t = a_t f(\phi_t) + b_t + \varepsilon_t,
\end{equation*}
i.e. the model (\ref{intro:observationequation})-(\ref{intro:acdstateequation}) with $\mu_a=1$, $\mu_b=0$ and unknown oscillation pattern $f$. We use ECG recordings obtained from the PhysioBank
database\footnote{\texttt{http://www.physionet.org/physiobank/}}.
The data are sampled at a frequency of 0.01 seconds for a duration of 10 seconds
(leading to 1000 observations) - see the top plot of Figure
\ref{fig:ecgDataEstimates}.
\begin{figure}[h!]
\centering
\includegraphics[width=380pt, height=320pt]{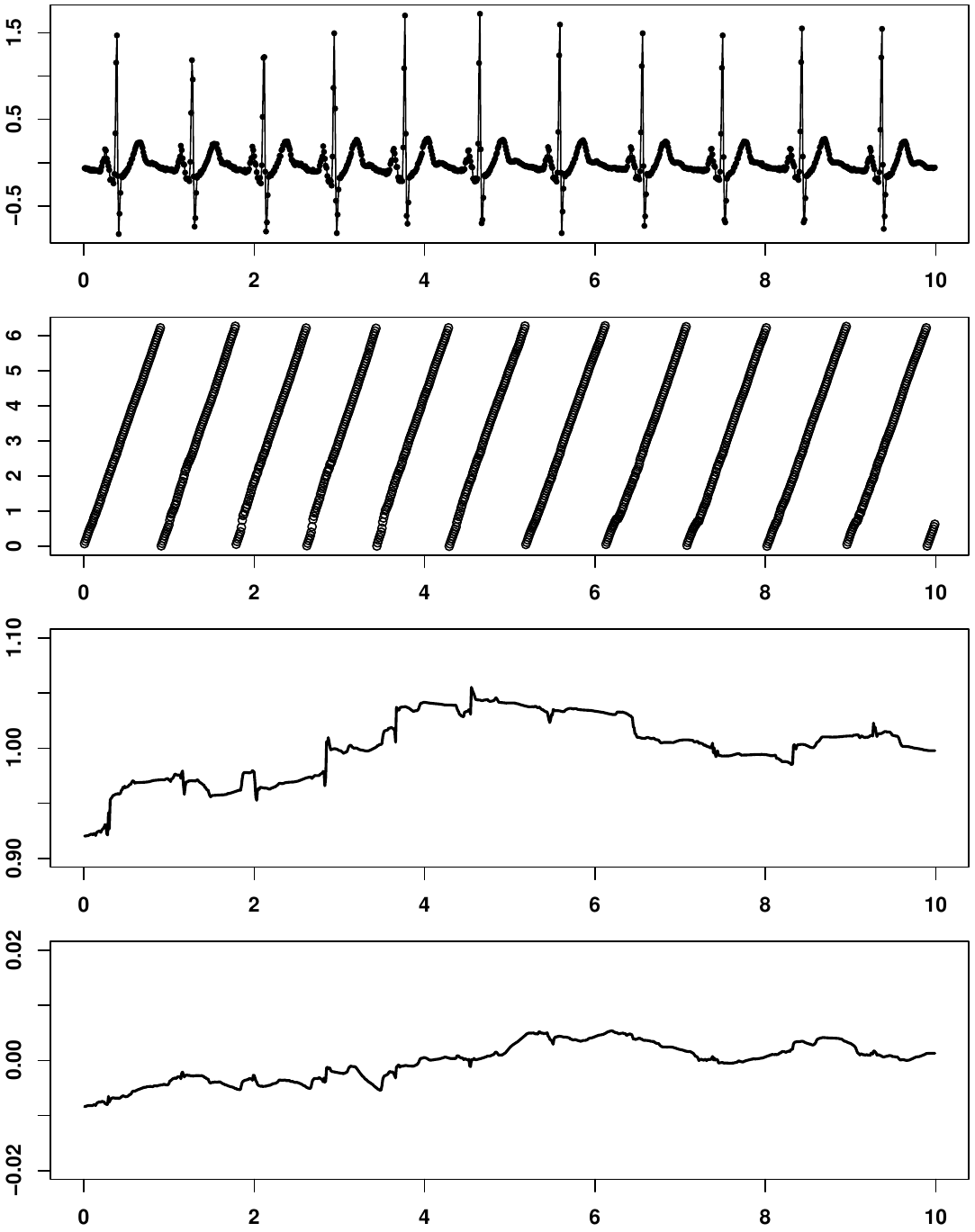}
\caption{\footnotesize Estimation results for the ECG recordings. The plots
show (from top to bottom): the ECG data points; the folded phase, the
amplitude, and the baseline estimated by the RBPS.}
\label{fig:ecgDataEstimates}
\end{figure}
The RBPS and the nonparametric EM algorithm are applied
to the data in order to obtain estimates for $\phi_t$, $a_t$, $b_t$, and $f$.
As initial oscillation pattern we use the trivial choice $\hat{f}^{(0)} \equiv
0$.
The only ``prior'' information we use is that the dataset covers roughly 11 cycles
leading an average increase of $\Delta{\phi_{t}}$ of about $2 \pi/90$. According to
this we choose as initial values in the first iteration step  $\beta^{(0)} = 0.1$ and
$\alpha^{(0)} = (1-\beta^{(0)}) 2 \pi /90$ (see (\ref{acd:mean}) and \ref{enumerate1:basiccycle}) in Section~\ref{sec:Identifiability}).
The estimates for the amplitude, baseline, and phase computed by the RBPS (applied with $N=100$ particles and $\ell=10$) are given in Figure
\ref{fig:ecgDataEstimates}.
\begin{figure}
\centering
\includegraphics[width=400pt,keepaspectratio]{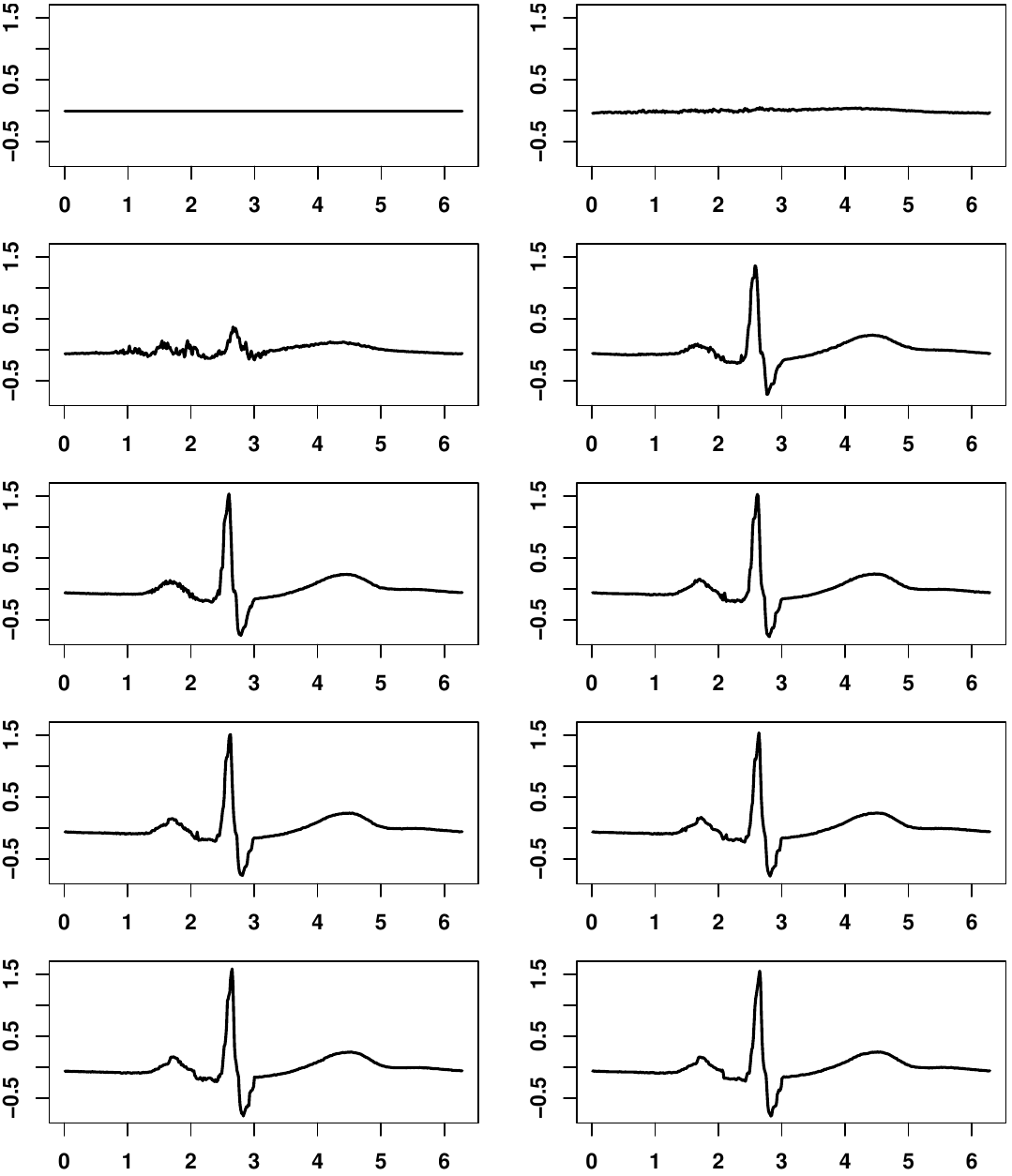}
\caption{\footnotesize The initial oscillation pattern $\hat{f}^{(0)}$ and
the estimated oscillation patterns $\check{f}^{(m)}$ for the iterations $m=1,
\ldots, 9$ of the nonparametric EM algorithm.}
\label{fig:ecgObsFuncEstC}
\end{figure}
It can be seen that the amplitude changes
significantly over time. In contrast, the baseline is almost constant.
The estimates of the oscillation
pattern $\check{f}^{(m)}$ for the iterations $m=1, \ldots, 9$ are
shown in Figure \ref{fig:ecgObsFuncEstC} (with $h=0.01$ and the Epanechnikov kernel in (\ref{NEM:optEstimator2})). Observe how rapidly the
estimates of the oscillation pattern converge. Finally, the estimated  $\check{f}^{(9)}$ is compared
with one period of the data in Figure \ref{fig:ecgObsFuncEst} (note that this is one oscillation and not the true curve). In this application we have chosen small $N$ and $l$ in order to demonstrate that this choice already leads to good results.

\begin{figure}
\centering
\includegraphics[width=395pt, height=85pt]{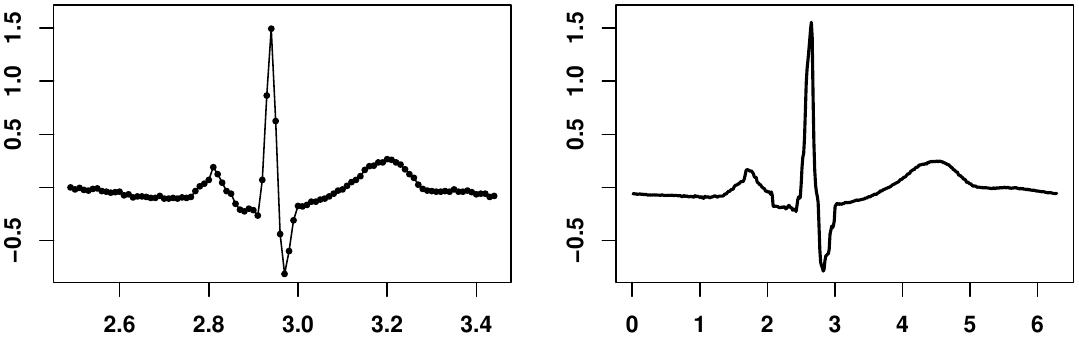}
\caption{\footnotesize Left plot: A fraction of the ECG recordings (note that this is one oscillation and not the true curve).
Right plot: Estimated oscillation pattern $\check{f}^{(9)}$ obtained after nine
iterations of the nonparametric EM algorithm.}
\label{fig:ecgObsFuncEst}
\end{figure}

% # Einstellungen im Algorithmus:
% # T=999 N=500
% # define SMOOTHING_LAG 40
% # define AMPLITUDE_NOISE_VARIANCE 0.0005
% # define BASELINE_NOISE_VARIANCE  0.00001
% # define AMPLITUDE_ALPHA 1.0
% # define BASELINE_ALPHA 1.0
% # double preAlpha = 0.2;
% # double preOmega = (1.0-preAlpha) * TWOPI / 90.0;
% # mProcessSigma = mOmega;
% # mObservationSigma = .25;

%In practice, the method could be used for denoising
%ECG recordings or the detection of anomalies caused by certain diseases
%(Clifford et.al. 2006).

\section{Concluding Remarks} \label{sec:Conclusion}

In this paper we have proposed a general model for oscillation processes with a quasi-periodic component.
The key ingredients are a nonparametric oscillation pattern $f$ and
the modeling of the unobserved phase process of the system by a nonlinear state space model. The situation
is challenging since the model is a nonparametric regression model with unobserved regressors. We have proven identifiability of the model for the specific phase model (\ref{StateEquation1}) which is the simplest possible phase model. We have chosen this model for a basic understanding but also as a framework for first stage mathematical investigations. The identifiability results of Section~\ref{sec:Identifiability} and Section~\ref{Appendix:Identifiability} have already been quite challenging for this model and the asymptotic properties of the nonparametric EM-algorithm introduced in Section~\ref{sec:NonparametricEstimate} are even in this model still unknown. Proving consistency and asymptotic normality of this nonparametric EM estimator is very challenging: using results similar to Olsson et al. (2008) a first step would be to prove that the MCEM - estimate of this paper is a valid approximation of the usual exact nonparametric EM algorithm (following Fort and Moulines, 2003, or Le Corff and Fort, 2013). Then, the asymptotic properties of the exact nonparametric EM algorithm had to be established for the present GSSM - which remains an open problem in this case since the state consists of an integrated process modulo $2 \pi$.

From an applied view we prefer the more general state equation (\ref{intro:acdstateequation}) which guarantees positivity of the phase increments and includes some dependence of consecutive phase increments. For this more general model we have estimated the unobserved phases by a computationally efficient
Rao-Blackwellized Particle Smoother (RBPS) which allows for simultaneous estimation of the amplitude, the baseline, and the phase. For the unknown oscillation pattern $f$ we have derived a nonparametric EM algorithm and applied it to ECG recordings. In another example the method has been  used for estimating the nonlinear phase of a noisy R\"{o}ssler attractor. The examples confirm that the RBPS provides in general a good estimate and maybe a worthwhile alternative to common methods for estimating the instantaneous frequency such as the Hilbert transform - in particular in cases of large observation noise. The good performance for noisy signals is due
to the fact that the observation noise is modeled explicitly.

In the ECG-application the algorithm needed about 30 minutes on a personal computer, i.e. 3 minutes per EM-step. We are optimistic that with massive parallelization over the particles and/or reduction of the lag of the fixed lag smoother the RBPS can be performed online. The EM algorithm may also be used pseudo-online
by always switching to new parameter values after a fixed number of time steps. A better option seems to develop a stochastic approximation EM algorithm as in Delyon et.al. (1999) for the present model.

\section{Appendix 1: Identifiability} \label{Appendix:Identifiability}

In this appendix we discuss identifiability of the model (\ref{BasicModel}),(\ref{StateEquation1}) in detail and prove Theorem~\ref{TheoremIdentifiability} and \ref{TheoremIdentifiability2} from Section~\ref{sec:Identifiability}. We first argue that we can assume without loss of generality in all proofs that the sequence $\{\idnonzero_i (f)\}_{i \in \mathbb{N}_+}$ from (\ref{DefKappas}) is setwise coprime.

\begin{proposition} \label{prop:coprime}
Suppose that $\{\idnonzero_i (f)\}_{i \in \mathbb{N}_+}$ is not setwise coprime. If $d$ is the greatest common divisor define $\bar{f}(x) := f\big(\frac{x}{d}\big)$ and $\bar{\phi}_t := d \phi_t$. Then $\bar{f}(\cdot)$ is $2\pi$-periodic with $c_{k}(\bar{f})\!=\!c_{kd}(f)$ and $\idnonzero_i(\bar{f})\!=\! \idnonzero_i(f) / d$, meaning that $\{\idnonzero_i (\bar{f})\}_{i \in \mathbb{N}_+}$ is setwise coprime. Furthermore, $\bar{\omega} = d \,\omega$ and $\bar{\sigma}_\eta = d \, \sigma_\eta$.
\end{proposition}

\begin{proof}
Since $d$ is a common divisor of $\{\idnonzero_i (f)\}_{i \in \mathbb{N}_+}$ we have
\begin{equation*} \label{}
f(x) = \sum_{k\in\mathbb{Z}}c_k(f)\,\mathrm{e}^{ikx} =  \sum_{k\in\mathbb{Z}}c_{kd}(f)\,\mathrm{e}^{ikd x}
\end{equation*}
and therefore
\begin{equation*} \label{}
\bar{f}(x) = f\bigg(\frac{x}{d}\bigg)  = \sum_{k\in\mathbb{Z}}c_{kd}(f)\,\mathrm{e}^{ikx},
\end{equation*}
i.e. $\bar{f}(\cdot)$ is $2\pi$-periodic with $c_{k}(\bar{f})\!=\!c_{kd}(f)$, $d \idnonzero_i(\bar{f})\!=\! \idnonzero_i(f)$. The rest is obvious.
\end{proof}

We now define the autocovariance function of $Y_{k}$ by
\begin{equation*}
\Gamma^{\omega,\sigma^2_{\eta}}_{f,\sigma^2_{\varepsilon}}(\ell) \eqdef \mathbb{E}\left[Y_{\ell}Y_{0}\right]-\mathbb{E}\left[Y_0\right]^2\;.
\end{equation*}

\begin{lemma} \label{lem:autocov}
We have
\begin{equation} \label{ExpressionCovariances}
\Gamma^{\omega,\sigma^2_{\eta}}_{f,\sigma^2_{\varepsilon}}(\ell) = \left\{
\begin{array}{ll}
2\sum_{k=1}^{\infty}\left|c_{k}(f)\right|^2\cos(k\ell\omega)\,\mathrm{e}^{-\ell k^2\sigma_{\eta}^2/2}, & \ell \in \mathbb{N}_+ \\[6pt]
\sigma_{\varepsilon}^2 + 2 \sum_{k=1}^{+\infty}|c_{k}(f)|^2, & \ell = 0\,.
\end{array} \right.
\end{equation}
\end{lemma}

\begin{proof}
Let $g_{\ell}(x)$ be the probability density function of $\phi_{\ell}-\phi_0 = \sum_{k=1}^{\ell}\eta_k + \omega \ell \sim \mathcal{N} (\omega \ell, \sigma_{\eta}^2 \ell)$. Since $f$ is real, we have with $c_{k}:=c_{k}(f)$ for all $k\ge 1\;$ $c_{-k} = \overline{c_{k}}$  and therefore for $\ell \in \mathbb{N}_+$
\begin{align*}
\Gamma^{\omega,\sigma^2_{\eta}}_{f,\sigma^2_{\varepsilon}}(\ell) &= \mathbb{E}\left[f(\phi_{\ell})f(\phi_0)\right]-\mathbb{E}\left[f(\phi_0)\right]^2 =  \sum_{k_1,k_2=-\infty}^{\infty} \!c_{k_1}c_{k_2}\mathbb{E}\left[\mathrm{e}^{ik_1\phi_{\ell}}\,\mathrm{e}^{ik_2\phi_0}\right]-|c_{0}|^2\\
&=  \frac{1}{2\pi}\sum_{k_1,k_2=-\infty}^{\infty} \!c_{k_1}c_{k_2}\int_{[0,2\pi]\times\mathbb{R}}\,\mathrm{e}^{ik_1x}\,\mathrm{e}^{ik_2(x+y)}g_{\ell}(y)\, \rmd x \,\rmd y-|c_{0}|^2\\
%&=  \sum_{k=-\infty}^{\infty}c_{\!-k}\,c_{k}\int_{\mathbb{R}}\,\mathrm{e}^{iky}g_{\ell}(y)\,\rmd %y-|c_{0}|^2\\
&=  \sum_{k=-\infty}^{\infty}\left|c_{k}\right|^2\,\mathrm{e}^{ik\ell\omega}\,\mathrm{e}^{-\ell k^2\sigma_{\eta}^2/2}-|c_{0}|^2 = 2\sum_{k=1}^{\infty}\left|c_{k}\right|^2\cos(k\ell\omega)\,\mathrm{e}^{-\ell k^2\sigma_{\eta}^2/2}\;.
\end{align*}
by using the characteristic function of the normal distribution. The derivation for $\Gamma^{\omega,\sigma^2_{\eta}}_{f,\sigma^2_{\varepsilon}}(0)$ is straightforward (use e.g. the Parseval equality).
\end{proof}

We now prove Theorem~\ref{TheoremIdentifiability} and at the same time the second relation from (\ref{RelationGamma2}).
\medskip

\noindent \textbf{Proof of Theorem~\ref{TheoremIdentifiability}.} \hfill \\
\noindent We set $c_{k}^{\star} := c_{k}(f_\star)$, $\idnonzero_i^{\star} := \idnonzero_i(f_\star)$, and $\idnonzero_i := \idnonzero_i(f)$.

\begin{enumerate}[1)]
\item {\bf Identifiability of $\sigma_{\eta \star}^2$ and definition of $\gamma$:}\label{sec:iden:sigma}\\
$\Gamma^{\omega,\sigma^2_{\eta}}_{f,\sigma^2_{\varepsilon}}(\ell)=\Gamma^{\omega_\star,\sigma^2_{\eta \star}}_{f_\star,\sigma^2_{\varepsilon \star}}(\ell)$ for all $\ell\ge 0$ implies
\[
\sum_{k=1}^{\infty}\left|c_{k}(f)\right|^2\cos(k\ell\omega)\,\mathrm{e}^{-\ell k^2\sigma^2/2}=\sum_{k=1}^{\infty}\left|c_{k}^{\star}\right|^2\cos(k\ell\omega_{\star})\,\mathrm{e}^{-\ell k^2\sigma_{\eta \star}^2/2}\;.
\]
By definition of $\idnonzero_1$ and $\idnonzero^{\star}_1$,  for all $\ell\ge 1$
\begin{multline} \label{eq:nphmm:lhsrhs}
|c_{\idnonzero_1}(f)|^2\cos(\idnonzero_1\ell\omega )+\sum_{k=\idnonzero_1+1}^{\infty}\left|c_{k}(f)\right|^2\cos(k\ell\omega)\,\mathrm{e}^{-\ell (k^2-\idnonzero_1^2)\sigma_\eta^2/2}\\
=\mathrm{e}^{-\ell/2 (\idnonzero^{\star 2}_1 \sigma_{\eta \star}^2-  \idnonzero_1^2 \sigma_\eta^2)}\bigg(|c_{\idnonzero^\star_1}^{\star}|^2\cos(\idnonzero^\star_1 \ell \omega_{\star})+\sum_{k=\idnonzero^\star_1+1}^{\infty}\left|c_{k}^{\star}\right|^2\cos(k\ell\omega_{\star})\,\mathrm{e}^{-\ell (k^2-\idnonzero^{\star 2}_1)\sigma_{\eta \star}^2/2}\bigg)\;.
\end{multline}
Whatever $\omega$, $\limsup_{\ell\to +\infty} \cos(\idnonzero_1 \omega \ell) =1$ \big(if $\idnonzero_1\omega\in\pi\mathbb{Q}$, the set $\left\{\cos(\idnonzero_1\omega\ell) ;\; \ell\ge 1\right\}$ is finite and $\cos(\idnonzero_1\omega\ell)$ equals one for infinitely many $\ell$; if $\idnonzero_1\omega \notin\pi\mathbb{Q}$, $\idnonzero_1\omega\mathbb{Z}+2\pi\mathbb{Z}$ is a dense subset of $\mathbb{R}$ and $[-1,1]$ is the set of limit points of $\left\{\cos(\idnonzero_1\omega\ell) ;\; \ell\ge 1\right\}$\big). Therefore,
the $\limsup$ of the left hand side of $\eqref{eq:nphmm:lhsrhs}$ is $|c_{\kappa_1}(f)|$ while the one of the right hand side is $0$ if $\kappa_1^{\star}\sigma_{\eta \star} > \kappa_1\sigma_\eta$ or $+\infty$ if $\kappa_1^{\star}\sigma_{\eta \star} < \kappa_1\sigma_\eta$.
Therefore, $\idnonzero_1^2 \sigma_\eta^2 = \idnonzero^{\star,2}_1 \sigma_{\eta \star}^2$ and $|c_{\idnonzero^\star_1}^{\star}|=|c_{\idnonzero_1}(f)|$. Define
\begin{equation} \label{DefinitionGamma}
\gamma \eqdef\frac{\idnonzero_1}{ \idnonzero^{\star}_1}  = \frac{\sigma_{\eta \star}}{\sigma_\eta}\;.
\end{equation}
\item {\bf Identifiability of $\omega_{\star}$ and $\{|c_k^{\star}|\}_{k\ge 0}$:}\label{sec:iden:omega}\\
We obtain from (\ref{eq:nphmm:lhsrhs}) with $\gamma \sigma_\eta =\sigma_{\eta \star}$
\begin{multline}
\label{eq:identifiability:step2}
|c_{\idnonzero_1}(f)|^2\cos(\idnonzero_1\ell\omega )+\sum_{k=\idnonzero_1+1}^{\infty}\left|c_{k}(f)\right|^2\cos(k\ell\omega)\,\mathrm{e}^{-\ell (k^2-\idnonzero_1^2)\sigma_\eta^2/2}\\
= |c_{\idnonzero^\star_1}^{\star}|^2\cos(\idnonzero^\star_1 \ell \omega_{\star})+\sum_{k=\idnonzero^\star_1+1}^{\infty}\left|c_{k}^{\star}\right|^2\cos(k\ell\omega_{\star})\,\mathrm{e}^{-\ell (k^2-\idnonzero^{\star,2}_1)\sigma_{\eta \star}^2/2} \;.
\end{multline}
Then, since $|c_{\idnonzero^\star_1}^{\star}|=|c_{\idnonzero_1}(f)|$,
\begin{align}
2\sin \Big(\frac{\gamma\omega+ \omega_{\star}}{2}\idnonzero^\star_1\ell\Big)\sin\Big(\frac{\gamma\omega- \omega_{\star}}{2}\idnonzero^\star_1\ell\Big)&  =
\cos(\idnonzero^\star_1 \ell \omega_{\star})-\cos(\idnonzero_1\ell\omega ) \label{eq:identifiability:sinus}\\[4pt]
& \hspace{-6.2cm}= \!\sum_{k=\idnonzero_1+1}^{\infty} \!\frac{\left|c_{k}(f)\right|^2}{|c_{\idnonzero^\star_1}^{\star}|^2}\cos(k\ell\omega)\,\mathrm{e}^{-\ell (k^2-\idnonzero_1^2)\sigma_\eta^2/2} -\sum_{k=\idnonzero^\star_1+1}^{\infty}\!\frac{\left|c_{k}^{\star}\right|^2}{|c_{\idnonzero^\star_1}^{\star}|^2}\cos(k\ell\omega_{\star})
\,\mathrm{e}^{-\ell(k^2-\idnonzero^{\star 2}_1)\sigma_{\eta \star}^2/2}.\label{eq:identifiability:omega}
\end{align}
If $\frac{\gamma\omega+ \omega_{\star}}{2} \idnonzero^\star_1 \notin \pi\mathbb{Z}$ and $\frac{\gamma\omega - \omega_{\star}}{2} \idnonzero^\star_1 \notin \pi\mathbb{Z}$, \eqref{eq:identifiability:sinus} does not converge to $0$ as $\ell \to \infty$ while \eqref{eq:identifiability:omega} does. Therefore there exists $k_1 \in \mathbb{Z}$ such that either $\gamma\omega+ \omega_{\star} = \frac {2\pi} {\idnonzero^\star_1} \, k_1$ or $\gamma\omega - \omega_{\star} = \frac {2 \pi} {\idnonzero^\star_1}\,k_1$. This implies $\cos(\idnonzero_1\ell\omega ) = \cos(\idnonzero^\star_1 \ell \omega_{\star})$ and, combined with  $|c_{\idnonzero_1}(f)| = |c_{\idnonzero^\star_1}^{\star}|$ and \eqref{eq:identifiability:step2}:
\begin{equation*}
 \sum_{k=\idnonzero_1+1}^{\infty}\left|c_{k}(f)\right|^2\cos(k\ell\omega)\,\mathrm{e}^{-\ell (k^2-\idnonzero_1^2)\sigma_\eta^2/2}
=  \sum_{k=\idnonzero^\star_1+1}^{\infty}\left|c_{k}^{\star}\right|^2\cos(k\ell\omega_{\star})\,\mathrm{e}^{-\ell (k^2-\idnonzero^{\star 2}_1)\sigma_{\eta \star}^2/2} \,.
\end{equation*}
Note, that $\kappa^\star_i$ is the $i^{th}$ non-zero coefficient of $c_k^\star$. Iterating these steps recursively proves for all $i \!\in \! \mathbb{N}_+\,$ $\idnonzero_i = \gamma \idnonzero^\star_i$ and
\begin{equation*}
 \sum_{k=\idnonzero_i+1}^{\infty}\left|c_{k}(f)\right|^2\cos(k\ell\omega)\,\mathrm{e}^{-\ell (k^2-\idnonzero_i^2)\sigma_\eta^2/2}
=  \sum_{k=\idnonzero^\star_i+1}^{\infty}\left|c_{k}^{\star}\right|^2\cos(k\ell\omega_{\star})\,\mathrm{e}^{-\ell (k^2-\idnonzero^{\star 2}_i)\sigma_{\eta \star}^2/2} \,.
\end{equation*}
This implies $|c_{\idnonzero_i}(f)|\!=\! |c_{\idnonzero^\star_i}^{\star}|$ for all $i \!\in \! \mathbb{N}_+$ (all other coefficients are $0$) and $\mean \left[Y_{k}\right] \!=\!c_{0} (f) \!=\! c_{0}^{\star}$.

We now show that $\gamma = 1$ if each of the sequences $\{\idnonzero_i\}_{i \in \mathbb{N}_+}$ and $\{\idnonzero_i^{\star}\}_{i \in \mathbb{N}_+}$ is setwise coprime (see Proposition~\ref{prop:coprime}). Then there exist $p\ge 1$ and $(a_1,\dots,a_p)'\in\mathbb{Z}^p$ such that
\begin{equation} \label{eq:sumcoprime}
\sum_{i=1}^p a_i\kappa_i^{\star} = 1\,.
\end{equation}
Therefore,
\[
\gamma = \gamma \sum_{i=1}^p a_i\kappa_i^{\star} =\sum_{i=1}^p a_i\kappa_i\in\mathbb{Z}\,.
\]
Since for all $i\ge 1\;$ $\gamma \kappa_i^{\star} = \kappa_i$, $\gamma$ is a divisor of all the $\kappa_i$'s implying $\gamma =1$.

We also obtain from above the existence of a sequence $\ell_i \in \mathbb{Z}$ such that either \linebreak $(\omega - \omega_{\star}) \idnonzero^\star_i \!=\! 2\pi \, \ell_i$ (*) or $(\omega + \omega_{\star}) \idnonzero^\star_i \!=\! 2 \pi\,\ell_i$ (**). Suppose first that (*) holds for some index $i$ and (**) holds for some index $j \neq i$. Then
$\omega_{\star} \!=\! \pi \big(\frac {\ell_j} {\idnonzero_j^{\star}} -\frac {\ell_i} {\idnonzero_i^{\star}}\big)$ which is in contradiction to $\omega_{\star} \notin \pi \mathbb{Q}$. This means that either (*) or (**) must hold simultaneously for all $i$.

Suppose now that (*) holds for all $i$. Then (\ref{eq:sumcoprime}) implies
\begin{equation*} \label{}
\omega - \omega_{\star} =  (\omega - \omega_{\star}) \sum_{i=1}^{p} a_i \idnonzero^\star_i \!=\! 2\pi \, \sum_{i=1}^{p} a_i \ell_i \in 2\pi \mathbb{Z}.
\end{equation*}
The assumption $\,\omega,\omega_{\star}  \!\in\! (0,\pi)$ then implies $\omega = \omega_{\star}$.

Finally, suppose that (**) holds for all $i$. This leads in the same way to $\omega + \omega_{\star} \in 2\pi \mathbb{Z}$ which is in contradiction with $\omega,\omega_{\star}  \!\in\! (0,\pi)$. Thus, we finally obtain $\omega = \omega_{\star}$.

\item {\bf Identifiability of $\sigma_{\varepsilon \star}^{2}$:}\label{sec:iden:sigmaeps}\\
The identifiability of $\sigma_{\varepsilon \star}^{2}$ then is a direct consequence of \eqref{ExpressionCovariances}.

\item {\bf Identifiability of $f_{\star}$:}\label{sec:iden:fstar}\\
We need to show that $c_k(f)=c_k^{\star}$ for all $k \in \mathbb{N}_0$. We know already know from \ref{sec:iden:omega}) that $c_0(f)=c_0^{\star}$ and $|c_k(f)|=|c_k^{\star}|$ for all $k \in \mathbb{N}_+$. The identification of the $c_k(f)$ for all $k\ge 0$ can be obtained using the higher order moments: for all $p\ge 1$ and all $(\ell_1,\dots,\ell_p)\in \mathbb{N}_{+}^p$, we define%\note {it seems to me that the case $p=2$ would be sufficient here?}\note {Here is a problem since we only know $\mathbb{E}Y_0 Y_{\ell_1}\cdots Y_{\ell_p}$, i.e. we have to identify the higher order moments of $\varepsilon_k$ as well. In case $p=2$ there only is a problem if $\ell_1=\ell_2=0$}\\
\begin{equation} \label{DefinitionLargePsi}
\Psi_{f_{\star},p}^{\omega_{\star},\sigma^2_{\eta \star}}(\ell_1,\dots,\ell_p)\eqdef \mathbb{E}\left[f_{\star}(\phi_0)f_{\star}(\phi_{\ell_1})\dots f_{\star}(\phi_{\ell_1+\dots+\ell_p})\right]\;.
\end{equation}
As the finite dimensional distributions of $\{Y_k\}_{k \in \mathbb{N}_+}$ are the same and the $\{\varepsilon_k\}_{k \in \mathbb{N}_+}$ are independent centered Gaussian random variables (implying that with $\sigma_{\varepsilon}^{2}$ also all moments  $\mathbb{E} \varepsilon_t^{\ell}$ are identifiable), we know that for all $p\ge 1$ and all $\boldsymbol{\ell}_p\in\mathbb{N}_{+}^p$ , $\Psi_{f,p}^{\omega,\sigma^2}(\boldsymbol{\ell}_p) = \Psi_{f_{\star},p}^{\omega_{\star},\sigma^2_{\eta \star}}(\boldsymbol{\ell}_p)$. Let further
\begin{equation} \label{eq:characteristic}
\psi^{\omega_{\star},\sigma^2_{\eta \star}}_{\ell}(k) \eqdef \mathrm{exp}\left(ik\ell\omega_{\star} - \frac{\sigma_{\eta \star}^2}{2}\ell k^2\right)\;
\end{equation}
be the characteristic function of a Gaussian random variable with mean $\ell \omega_{\star}$ and variance $\ell \sigma^2_{\eta \star}$. Write, for all $-\infty< k_1,\dots, k_p< \infty$,
\begin{align*}
d^{\star}_{k_1,\dots,k_p}&\eqdef c^{\star}_{-(k_1+\dots+k_p)}\,c^{\star}_{k_1}\dots c^{\star}_{k_p}\\d_{k_1,\dots,k_p}&\eqdef c_{-(k_1+\dots+k_p)}(f)\,c_{k_1}(f)\dots c_{k_p}(f)\;.
\end{align*}
We know from \ref{sec:iden:omega}) that $\sigma^2=\sigma_{\eta \star}^2$ and $\omega = \omega_{\star}$. Then for all $p\ge 1$ and all $\boldsymbol{\ell}_p\in\mathbb{N}_{+}^p$, $\Psi_{f_{\star},p}^{\omega_{\star},\sigma^2_{\eta \star}}(\boldsymbol{\ell}_p) =\Psi_{f,p}^{\omega_\star,\sigma^2_{\eta \star}}(\boldsymbol{\ell}_p)$ implies because of Lemma~\ref{lem:psi}
\[
\sum_{-\infty< k_1,\dots,k_p<\infty}\psi^{\omega_{\star},\sigma^2_{\eta \star}}_{\ell_1}(k_1+\dots+k_p)\dots\psi^{\omega_{\star},\sigma^2_{\eta \star}}_{\ell_p}(k_p)\left(d^{\star}_{k_1,\dots,k_p}-d_{k_1,\dots,k_p}\right)=0\;.
\]
Notice that, since $\{c^\star_k\}_{k\ge 0}$ and $\{c_k\}_{k\ge 0}$ belong to $\ell_2(\mathbb{Z})$,  the coefficients $d^{\star}_{k_1,\dots,k_p}-d_{k_1,\dots,k_p}$, $k_1,\ldots,k_p \in\mathbb{Z}$ are bounded. Therefore, by Lemma~\ref{lem:psi:multiple:invert}, for all $-\infty< k_1,\dots,k_p<\infty$, $d^{\star}_{k_1,\dots,k_p}=d_{k_1,\dots,k_p}$.
Due to (\ref{eq:sumcoprime}) we can decompose any  $k\in\mathbb{Z}$ as $k=ka_1\idnonzero_1+\dots+ka_p\idnonzero_p $. Thus  $d^{\star}_{\idnonzero_1,\dots,\idnonzero_p}=d_{\idnonzero_1,\dots,\idnonzero_p}$ yields
\[
c^{\star}_{-k}(c^{\star}_{\idnonzero_1})^{a_1k}\dots (c^{\star}_{\idnonzero_p})^{a_pk}=c_{-k}(f)(c_{\idnonzero_1}(f))^{a_1k}\dots (c_{\idnonzero_p}(f))^{a_pk}\;.
\]
Therefore,
\[
c_{-k}(f) = c^{\star}_{-k} \left[\frac{(c^{\star}_{\idnonzero_1})^{a_1}\dots (c^{\star}_{\idnonzero_p})^{a_p}}{(c_{\idnonzero_1}(f))^{a_1}\dots (c_{\idnonzero_p}(f))^{a_p}}\right]^k\;.
\]
As, for all $1\le i \le d$ $|c_{\idnonzero_i}(f)| = |c^{\star}_{\idnonzero_i}|\neq 0$, there exists $\theta \in [0,2\pi)$ such that
\[
c_{-k}(f) = c^{\star}_{-k}\,\mathrm{e}^{-ik\theta}\;.
\]
which completes the proof.  \hfill $\Box$
\end{enumerate}

\noindent \textbf{Proof of Theorem~\ref{TheoremIdentifiability2}.} \hfill \\
\noindent Let $f_{\star}$ be a non-constant oscillation pattern and
\begin{equation} \label{DefinitionReplication3}
\nu(f_{\star}) =  \max \big\{ j \in \mathbb{N}_+ \,\big| \,c_k(f_{\star}) = 0 \;\; \forall \; k \neq j\mathbb{N}_+ \big\}\,.
\end{equation}
Note that $\nu(f_{\star}) < \infty$ (otherwise $f_{\star}$ were constant). We now prove that $\nu(f_{\star}) = \textnormal{\small{repl}}(f_{\star})$. Since $c_k(f_{\star}) = 0$ apart from $k = \ell \, \nu(f_{\star})$, with some $\ell \in \mathbb{Z}$,
we obtain
\begin{equation*} \label{}
\bar{f_{\star}}(x) := f_{\star}\bigg(\frac{x}{\nu(f_{\star})}\bigg)  = \sum_{\ell\in\mathbb{Z}}c_{\ell \, \nu(f_{\star})}(f_{\star})\,\mathrm{e}^{i (\ell \, \nu(f_{\star})) \frac {x} {\nu(f_{\star})}} = \sum_{\ell\in\mathbb{Z}}c_{\ell \, \nu(f_{\star})}(f_{\star})\,\mathrm{e}^{i \ell \,x}\,,
\end{equation*}
implying that $\bar{f_{\star}}$ is $2\pi$-periodic and $\textnormal{\small{repl}}(f_{\star}) \ge \nu (f_{\star})$. The sequence $\{\kappa_i (\bar{f_{\star}})\}_{i \in \mathbb{N}_+}$ must be setwise coprime (if the sequence had a common factor $m$ also $m \, \nu(f_{\star})$ would fulfill the above requirement and $\nu(f_{\star})$ were not the maximum).

Suppose now that $f$ is another oscillation pattern and $\nu(f)$ and $\bar{f}$ are defined as above. Then it follows from \ref{sec:iden:omega}) above that there exists a $\gamma$ with $\kappa_i (\bar{f}) = \gamma \kappa_i (\bar{f_{\star}})$ for all $i \in \mathbb{N}_+$. Since both sequences
$\{\kappa_i (\bar{f_{\star}})\}_{i \in \mathbb{N}_+}$ and $\{\kappa_i (\bar{f})\}_{i \in \mathbb{N}_+}$ are setwise coprime it follows as in \ref{sec:iden:omega}) that $\gamma=1$. As in \ref{sec:iden:fstar}) we therefore obtain $\bar{f}(x-\theta) = \bar{f}_{\star}(x)$ with some $\theta \in (0,2\pi)$ and $f(x) = \bar{f}_{\star}\big(\nu(f)\, x + \theta\big)$. In particular we have
$\nu (f_{\star}) = \textnormal{\small{repl}}(f_{\star})$ and the basic cycle is unique.

Since $f_{\star}\big(\frac{x}{\nu(f_{\star})}\big) = \bar{f_{\star}}(x) = \bar{f}(x-\theta) = f \big(\frac{x-\theta}{\nu(f)}\big)$ we obtain from Theorem~\ref{TheoremIdentifiability}
\begin{equation*} \label{}
\gamma = \frac {\nu(f)} {\nu (f_{\star})} = \frac {\textnormal{\small{repl}}(f)} {\textnormal{\small{repl}}(f_{\star})}. \tag*{$\Box$}
\end{equation*}

\subsection{Technical Lemmata} \label{Appendix:TechnicalLemmata}
The proof of the following results are not difficult and are therefore omitted.
\begin{lemma}\label{lem:psi}
Let $\Psi_{f_{\star},p}^{\omega_{\star},\sigma^2_{\eta \star}}(\ell_1,\dots,\ell_p)$ and $\psi^{\omega_{\star},\sigma^2_{\eta \star}}_{\ell}$ be as in (\ref{DefinitionLargePsi}) and (\ref{eq:characteristic}) respectively. Then for all $(\ell_1,\dots,\ell_p)\in \mathbb{N}_{+}^p$,
\begin{equation}
\label{eq:psi}
\Psi_{f_{\star},p}^{\omega_{\star},\sigma^2_{\eta \star}}(\ell_1,\dots,\ell_p)  = \!\!\!\!\!\!\!\! \sum_{-\infty< k_1,\dots,k_p<\infty} \!\!\!\! \!\!\!\! c^{\star}_{-(k_1+\dots+k_p)}\,c^{\star}_{k_1}\dots c^{\star}_{k_p}\,\psi^{\omega_{\star},\sigma^2_{\eta \star}}_{\ell_1}(k_1+\dots+k_p)\dots\psi^{\omega_{\star},\sigma^2_{\eta \star}}_{\ell_p}(k_p)\;.
\end{equation}
\end{lemma}

%\begin{proof}
%Let $\boldsymbol{\ell}_p = (\ell_1,\dots,\ell_p)$.
%\begin{align*}
%\Psi_{f_{\star},p}^{\omega_{\star},\sigma^2_{\eta \star}}(\boldsymbol{\ell}_p) &= %\mathbb{E}\left[f(\phi_0)f(\phi_{\ell_1})\dots f(\phi_{\ell_1+\dots+\ell_p})\right]\\
%%&= \sum_{-\infty< k_0,\dots,k_p< \infty} c^{\star}_{k_0}\dots %c^{\star}_{k_p}\mathbb{E}\left[\mathrm{e}^{ik_0\phi_0}\dots %\mathrm{e}^{ik_p\phi_{\ell_1+\dots+\ell_p}}\right]\\
%&= \sum_{-\infty< k_0,\dots,k_p< \infty} c^{\star}_{k_0}\dots c^{\star}_{k_p}\frac{1}{2\pi}\int %\mathrm{e}^{ik_0x_0}\dots %\mathrm{e}^{ik_p\left(x_0+\dots+x_p\right)}\prod_{i=1}^{p}g_{\ell_i}(x_i)\,\rmd x_{0:p}\\
%&= \sum_{-\infty< k_0,\dots,k_p< \infty} c^{\star}_{-(k_1+\dots+k_p)}c^{\star}_{k_1}\dots %c^{\star}_{k_p}\int \mathrm{e}^{i(k_1+\dots+k_p)x_1}\dots %\mathrm{e}^{ik_px_p}\prod_{i=1}^{p}g_{\ell_i}(x_i)\,\rmd x_{1:p}\\
%&= \sum_{-\infty< k_0,\dots,k_p< \infty} c^{\star}_{-(k_1+\dots+k_p)}c^{\star}_{k_1}\dots %c^{\star}_{k_p}\psi^{\omega_{\star},\sigma^2_{\eta %\star}}_{\ell_1}(k_1+\dots+k_p)\dots\psi^{\omega_{\star},\sigma^2_{\eta \star}}_{\ell_p}(k_p)\;.
%\end{align*}
%\end{proof}

\begin{lemma}
\label{lem:psi:invertible}
Let $\{z_k\}_{k=-\infty}^\infty$ be complex numbers such that $\{z_k\}_{k=-\infty}^\infty \in \ell_\infty(\mathbb{Z})$. Then, for all $\ell\ge 1$, $\{ \psi^{\omega_{\star},\sigma^2_{\eta \star}}_{\ell}(k)z_k\}_{k\in\mathbb{Z}} \in \ell_1(\mathbb{Z})$, where $\psi^{\omega_{\star},\sigma^2_{\eta \star}}_{\ell}$ is defined in \eqref{eq:characteristic}. Assume that $\omega_{\star}\notin \pi\mathbb{Q}$. If for all $\ell \ge 1$,
\[
\sum_{-\infty<k<\infty} \psi^{\omega_{\star},\sigma^2_{\eta \star}}_{\ell}(k)z_k=0\;,
\]
then we have $z_k=0$ for all $ k\in\mathbb{Z}$.
\end{lemma}

\noindent The proof of Theorem~\ref{TheoremIdentifiability} 4) relies on the following lemma which is a  corollary of Lemma~\ref{lem:psi:invertible}.

\begin{lemma}
\label{lem:psi:multiple:invert}
Assume that $\omega_{\star}\notin \pi\mathbb{Q}$. Let  $Z \in \ell_\infty(\mathbb{Z}^p)$ be a bounded complex sequence indexed by $(k_1,\ldots,k_p) \in \mathbb{Z}^p$, $p\ge1$, satisfying, for any $\ell_1,\cdots,\ell_p\ge 1$,
\[
\sum_{-\infty< k_1,\dots,k_p<\infty}\psi^{\omega_{\star},\sigma^2_{\eta \star}}_{\ell_1}(k_1+\dots+k_p)\dots\psi^{\omega_{\star},\sigma^2_{\eta \star}}_{\ell_p}(k_p)z_{k_1,\ldots,k_p}=0\;.
\]
Then we have $z_{k_1,\cdots,k_p}=0$ for any $k_1,\cdots,k_p$.
\end{lemma}

\section{Appendix 2: The nonparametric EM-estimate} \label{sec:Appendix2}

\noindent
{\bf Proof of Proposition \ref{NEM:prop1}.}
Under the assumption that $f$ is $2\pi$ periodic, it can be seen from
(\ref{NonparaEM}) that $\sum_{t=1}^T \int \big\{y_t - a_t f(\phi_t\, \text{mod} \, 2\pi) -
b_t\big\}^2 p_{f^{(m)}}(a_t, b_t, \phi_t| y_{1:t+\ell}) \,\rmd a_t \,\rmd b_t \,\rmd \phi_t$  needs to be minimized with respect to $f(\phi)$. By using the
density $p(a_{t}, b_{t}, \phi_{t}| y_{1:t+\ell})$ from (\ref{kernelprobdistribution}) the above expression becomes
\begin{align*} \label{}
&\sum_{t=1}^T  \sum_{i=1}^N \tilde{\omega}_t^i \,  \int_{0}^{2\pi} \!\!\int \!\!\int \, \Big(y_t -
a_t f(\phi) - b_t\Big)^{2}
\mathcal{N}\Big(a_t, b_t \big| (\tilde{a}_t^i, \tilde{b}_t^i)^T,
\tilde{\Sigma}_t^i \Big)\, K_{h}\big( (\phi-\phi_{t}^i) \mbox{ mod } 2\pi \big) \,\rmd a_t \,\rmd b_t \,\rmd \phi\\
& \qquad = \int_{0}^{2\pi} \sum_{t=1}^T  \sum_{i=1}^N  \tilde{\omega}_t^i \,  \Big(y_t^{2} +
(\tilde{S}_t^i)_{11} f(\phi)^{2} + (\tilde{S}_t^i)_{22} - 2 y_t
\tilde{a}_t^i f(\phi) - 2 y_t
\tilde{b}_t^i + 2 (\tilde{S}_t^i)_{12} f(\phi) \Big)\\
& \hspace*{10.0cm} \times K_{h}\big( (\phi-\phi_{t}^i) \mbox{ mod } 2\pi \big) \, \rmd \phi.
\end{align*}
For fixed $\phi$ minimization with respect to $f(\phi)$ now yields
\begin{equation*} \label{}
\tilde{f}^{(m+1)}(\phi) =
\frac{
\sum_{t=1}^T \sum_{i=1}^N
\tilde{\omega}_t^i \, K_{h}\big( (\phi-\phi_{t}^i) \mbox{ mod } 2\pi \big) \big\{ y_t
\tilde{a}_t^i - (\tilde{S}_t^i)_{12} \big\}
}{
\sum_{t=1}^T
\sum_{i=1}^N
\tilde{\omega}_t^i \,K_{h}\big( (\phi-\phi_{t}^i) \mbox{ mod } 2\pi \big)
(\tilde{S}_t^i)_{11} } \,.
\end{equation*}
i.e. the result. For the filter and the smoother the proofs are the same.  \hfill $\Box$

\vspace{0.5cm}

\noindent \textbf{Acknowledgements:} This research has been supported by the Deutsche Forschungsgemeinschaft under DA 187/15-1, through the Research Training Group RTG 1953 and by Heidelberg University under Frontier D.801000/08.023. It has been conducted as part of the project Labex MME-DII (ANR11-LBX-0023-01). The authors are very grateful to two anonymous referees for their helpful comments.

\section*{References}

\begin{description}
\baselineskip1.3em
\itemsep-0.04cm

\item Amblard, P. O., Brossier, J. M., and Moisan, E. (2003). Phase tracking: what do we gain from optimality? Particle filtering versus phase-locked loops. {\it Signal Processing} 83, 151-167.

\item Blasius, B., Huppert, A., and Stone, L. (1999) Complex dynamics and phase synchronization in spatially extended ecological
systems. {\it Nature} 399, 354-359.

\item Chen, R.-R., Koetter, R., Madhow, U., and Agrawal, D. (2003) Joint
noncoherent demodulation and decoding for the block fading channel: A
practical framework for approaching Shannon capacity. In  {\it IEEE
Transactions on Communications} 51, 1676-1689.

\item Clifford, G.D., Azuaje, F., and McSharry, P.E. (2006) {\it Advanced
Methods and Tools for ECG Data Analysis}. Norwood: Artech House.

\item Delyon, B., Lavielle, M., and Eric Moulines, E. (1999) Convergence of a stochastic approximation version
of the EM algorithm. {\it  Annals
of Statistics} 27, 94--128.

\item Dempster, A.P., Laird, N.M., and Rubin, D.B. (1977) Maximum Likelihood from Incomplete Data via the EM Algorithm. {\it
Journal of the Royal Statistical Society, Series B} 39, 1-38.

\item Douc, R., Capp{\'{e}}, O., and Moulines, E. (2005) Comparison of
resampling schemes for particle filtering. In  {\it Proceedings of the 4th
International Symposium on Image and Signal Processing and Analysis}, 64-69.

\item Douc, R., Moulines, E. and Stoffer, D.S. (2014) {\it Nonlinear time series: Theory, Methods and Applications with R Examples}. CRC Press.

\item Doucet, A., Gordon, N.J., and Krishnamurthy, V. (1999) Particle Filters for State Estimation of Jump Markov Linear Systems. {\it
IEEE Transactions on Signal Processing} 49, 613-624.

\item Doucet, A., Godsill, S. and Andrieu, C. (2000a) On sequential Monte Carlo sampling methods for Bayesian filtering. {\it Stat. Comput.} 10, 197-208.

\item Doucet, A., De Freitas, N., Murphy, K. and Russell, S. (2000b) Rao-Blackwellised particle filtering for dynamic Bayesian networks. In  {\it Proceedings of the Sixteenth conference on uncertainty in artificial intelligence}, 176-183.

\item Doucet, A., De Freitas, N., and Gordon, N. (ed.) (2001)
{\it Sequential Monte Carlo Methods in Practice}, Springer Verlag, New York.

\item Dubois, C. and Davy, M. (2007) Joint detection and tracking of time-varying harmonic components: a flexible Bayesian approach. {\it IEEE Trans. Audio, Speech, Language Process.} 15(4), 1283–1295.

\item Dumont, T. and Le Corff, S. (2014)
{\it Nonparametric regression on hidden $\Phi$-mixing variables: identifiability and consistency of a pseudo-likelihood based estimation procedure}. arXiv:1209.0633.

\item Engle, R.F., and Russell, J.R. (1998) Autoregressive conditional duration: A new model for irregularly spaced
transaction data. {\it Econometrica} 66, 1127-1162.

\item Fell, J., and Axmacher, N. (2011) The role of phase synchronization in memory processes. {\it Nat. Rev. Neurosci.} 12(2), 105-118.
    
\item Fort, G., and Moulines, E. (2003) Convergence of the Monte Carlo expectation maximization for curved exponential families. {\it Annals of Statistics} 31 (4), 1220-1259.

\item Godsill, S.J., Doucet, A., and West, M. (2004) Monte Carlo Smoothing
for Nonlinear Time Series. {\it Journal of
American Statistical Association} 99, 156-168.

\item Gordon, N. J., Salmond, D. J., and Smith, A. F. (1993) Novel approach to nonlinear/non-Gaussian Bayesian state estimation. {\it IEE Proceedings F-Radar and Signal Processing} 140 (2), 107-113.

\item Grossmann, A. , Kronland-Martinet, R., and Morlet, J. (1989) Reading and Understanding Continuous Wavelet Transforms. {\it Wavelets,
Time-Frequency Methods and Phase Space}, ed. J.M. Combes, Springer Verlag, Berlin.

\item Hannan, E.J. (1973) The Estimation of Frequency. {\it Journal of
Applied Probability} 10, 513-519.

\item Huang, N.E., Shen, Z., Long, S.R., Wu, M.L., Shih, H.H., Zheng, Q., Yen, N.C., Tung, C.C., and
Liu, H.H. (1998) The empirical mode decomposition
and Hilbert spectrum for nonlinear and nonstationary
time series analysis. {\it Proc. Roy. Soc.
London A} 454, 903-995.

\item H\"{u}rzeler, M., and K\"{u}nsch, H. (2001) Approximating and maximising the Likelihood for a General State-Space Model. {\it Sequential Monte Carlo Methods in Practice}, eds. A. Doucet, N. de Freitas and N. Gordon, New York: Springer.

\item Kitagawa, G. (1996) Monte Carlo filter and smoother for non-Gaussian nonlinear state space models. {\it Journal of Computational and Graphical Statistics} 5(1), 1-25.

\item Kneip, A. and Gasser, Th. (1992) Statistical Tools to Analyze Data Representing a Sample of Curves. {\it  The Annals of Statistics} 20, 1266-1305.

\item Kong, A., Liu, J., and Wong, W. (1994) Sequential imputation and Bayesian missing data problems. {\it Journal of
American Statistical Association} 89, 278-288.

\item Le Corff, S., and Fort, G. (2013) Online Expectation Maximization based algorithms for inference in hidden Markov models. {\it Electronic Journal of Statistics} 7, 763-792.

\item Lepski, O.V., Mammen, E., and Spokoiny, V.G. (1997) Optimal spatial adaptation to inhomogeneous smoothness: an approach based on kernel estimates with variable bandwidth selectors. {\it The Annals of Statistics} 25, 929-947.

\item Li, T.H. (2013) {\it Time Series with Mixed Spectra}, CRC Press.

\item Lloyd, A.L., and May, R.M. (1999) Synchronicity, chaos and population cycles: spatial coherence in an
uncertain world. {\it Trends in Ecology \& Evolution} 14, 417-418.

\item Myers, C.S. and L. R. Rabiner, L.R. (1981) A comparative study of several dynamic time-warping algorithms for connected word recognition. {\it The Bell System Technical Journal} 60, 1389-1409.

\item Olsson, J., Cappe, O., Douc, R. and Moulines, E. (2008) Sequential Monte Carlo smoothing with application to parameter estimation in nonlinear state space models. {\it Bernoulli} 14, 155-179.

\item Paraschakis, K. and Dahlhaus, R. (2012) Frequency and phase estimation in time series with quasi periodic components. {\it Journal of Time Series Analysis} 33, 13-31 .

\item Pikovsky, A.S., Rosenblum, M.G., Osipov, G.V., and Kurths, J. (1997) Phase synchronization of chaotic oscillators by external driving. {\it
Physica D} 104, 219-238.

\item Pikovsky, A., Rosenblum, M., and Kurths, J. (2001), {\it
Synchronization}. Cambridge University Press.

\item Press, W.H., Teukolsky, S.A., Vetterling, W.T., and Flannery, B.P. (1992),
{\it Numerical Recipes in C (2nd ed.)}. Cambridge University Press.

\item Quinn, B.G. and Hannan, E.J. (2001) {\em The Estimation and
Tracking of Frequency}. Cambridge University Press.

\item Rosenblum, M.G., Pikovsky, A.S., and Kurths, J. (1996) Phase Synchronization of Chaotic Oscillators. {\it Physical Review
Letters} 76, 1804-1807.

\item Schuster, A. (1897) On lunar and solar periodicities of earthquakes. {\it Proc. Roy. Soc.} 61, 455-465.

\item Shumway, R.H. and Stoffer, D.S. (1982) An approach to time series smoothing and forecasting using the EM algorithm. {\it  J. Time Series Anal.} 3, 253--264.

\item Shumway, R.H. and Stoffer, D.S. (2011) {\it Time Series Analysis and Its Applications: With R Examples (3rd ed)}.  Springer Verlag, New York.

\item Tanner, M.A. (1993) {\it Tools for Statistical Inference: Methods for the Exploration of Posterior Distributions and Likelihood Functions (2nd ed.)}.  Springer Series in Statistics, New York.

\item Tsakonas, E., Sidiropoulos, N. and Swami, A. (2008) Optimal particle filters for tracking a time-varying harmonic or chirp signal. {\it IEEE Trans. Signal Process.} 56(10), 4598-4610.

\item Wang, K. and Gasser, Th. (1997) Alignment of curves by dynamic time warping. {\it  Annals of Statistics}, 25, 1251-1276.

\item  Wei, G. C. G. and Tanner, M. A. (1990) A Monte Carlo implementation of the EM algorithm and the poor man's data augmentation algorithms. {\it Journal of the American Statistical Association} 85, 699-704.

\end{description}

\end{document}